\documentclass[12pt,onecolumn]{IEEEtran}
\everymath{\displaystyle}
\IEEEoverridecommandlockouts
\usepackage{cite}
\usepackage{graphicx}      % include this line if your document contains figures
\usepackage{epsfig}
\usepackage{float}
\usepackage{amsmath}
\usepackage{amssymb}
\usepackage{amsthm}
\usepackage{booktabs}
\usepackage{subfigure}
\usepackage{bm}
\usepackage{geometry}

\newcommand{\beqnum}{\begin{equation}\begin{array}{lcl}}
\newcommand{\eeqnum}{\end{array}\end{equation}}
\newcommand{\beqnom}{\begin{eqnarray}}
\newcommand{\eeqnom}{\end{eqnarray}}
\newcommand{\beqnc}{\begin{center}\begin{eqnarray}}
\newcommand{\eeqnc}{\end{eqnarray}\end{center}}
\newcommand{\beqnlm}{\begin{equation}\vspace{-.5cm}\begin{array}{lll}}
\newcommand{\eeqnlm}{\end{array}\end{equation}}\vspace{-.5cm}
\newcommand{\beq}{\begin{eqnarray*}}
\newcommand{\eeq}{\end{eqnarray*}}
\newcommand{\bef}{\begin{figure}[tbh!]}
\newcommand{\enf}{\end{figure}}

\newtheorem{defn}{\bf Definition}

\newtheorem{montheo}{\bf Theorem}
\newtheorem{rem}{\bf Remark}
\newtheorem{lemme}{\bf Lemma}
\newtheorem{Proposition}{\bf Proposition}

\title{Stabilization of perturbed integrator chains using Lyapunov-Based Homogeneous Controllers}
\author{Salah Laghrouche, Mohamed Harmouche,\\ and Yacine Chitour
\thanks{S. Laghrouche and M. Harmouche are with IRTES-SET Laboratory, UTBM, Belfort, France. (e-mail: \{salah.laghrouche,mohamed.harmouche\}@utbm.fr).}
\thanks{Y. Chitour is with L2S, Universite Paris XI, CNRS 91192 Gif-sur-Yvette, France. (e-mail:yacine.chitour@lss.supelec.fr)}}

\begin{document}
\maketitle
\thispagestyle{empty}
\pagestyle{empty}
\maketitle

%%%%%%%%%%%%%%%%%%%%%%%%%%%%%%%%%%%%%%%%%%%%%%%%%%%%%%%%%%%%%%%%%%%%%%%%%%%%%%%%
\begin{abstract}
\noindent In this paper, we present a Lyapunov-based homogeneous controller for the stabilization of a perturbed chain of integrators of arbitrary order $r\geq 1$. The proposed controller is based on homogeneous controller for stabilization of pure integrator chains. The control of homogeneity degree is also introduced and various controllers are designed using this concept, namely a bounded-controller with minimum amplitude of discontinuous control and a controller with fixed-time convergence. The performance of the controller is validated through simulations.
\end{abstract}
%%%%%%%%%%%%%%%%%%%%%%%%%%%%%%%%%%%%%%%%%%%%%%%%%%%%%%%%%%%%%%%%%%%%%%%%%%%%%%%%

%\begin{keywords}
%Global tracking, bounded feedback, Lyapunov function, underactuated surface marine vessels.
%\end{keywords}

%%%%%%%%%%%%%%%%%%%%%%%%%%%%%%%%%%%%%%%%%%%%%%%%%%%%%%%%%%%%

%%%%%%%%%%%%%%%%%%%%%%%%%%%%%%%%%%%%%%%%%%%%%%%%%%%%%%%%%%%%%%%%%%%%%%%%%%%%%%%%
%
\section{Introduction}

The problem of finite-time stabilization of a perturbed integrator chain arises in many control applications. For example, electromechanical systems such as motorized actuators or robotic arms are modeled as perturbed double integrators \cite{Baht1998,Hong2002_Robot,Orlov2009}. Another application is in Higher Order Sliding Mode Control (HOSM) \cite{Emelyanove}, which can be formulated as the stabilization of an auxiliary system arising as a perturbed integrator chain built from the output and its higher time derivatives \cite{Dinuzzo09}. The finite-time stability problem was addressed in relation with homogeneous systems in \cite{Bhat_Bernstein_1}, and homogeneity concept was used for stabilization of linear systems in \cite{Praly_1997}. In \cite{Bernstein2005}, the link between finite-time stabilization and homogeneity of a system was established, and it was shown that a homogeneous system is finite-time stable if and only if it is asymptotically stable and has a negative homogeneity degree. This result has, since then, been used for the development of many controllers for pure and perturbed integrator chains. A homogeneous nonsmooth proportional-derivative controller for robot manipulators (double integrator system) was developed in \cite{Hong2002_Robot}. This work was generalized for an arbitrary-length integrator chain in \cite{Hong}. In \cite{Qian_1} and \cite{Qian_2}, negative homogeneity was used for the finite-time stabilization of a class of nonlinear systems that includes perturbations at each integrator link.

Among Sliding Mode techniques, the homogeneity approach was used in \cite{Levant2003,Levant2005}, to demonstrate finite-time stabilization of the arbitrary order sliding mode controllers for Single Input Single Output (SISO) systems \cite{Levant2001}. A robust Multi Input Multi Output (MIMO) HOSM controller was also presented in \cite{Defoort2009}, using a constructive algorithm with geometric homogeneity based finite-time stabilization of an integrator chain. A controller, which stabilizes a perturbed integrator chain of arbitrary length using only the signs of state variables, was presented in \cite{Kryachkov}. A Lyapunov-based approach for arbitrary HOSMC controller design was first presented in \cite{Harmouche_CDC12}. In these works, it was shown that a class of homogeneous controllers that satisfies certain conditions, could be used to stabilize perturbed integrator chains.

In this paper, we present a continuation of \cite{Harmouche_CDC12}, and develop a Lyapunov-based robust controller for the finite-time stabilization of a perturbed integrator chain of arbitrary order, with bounded uncertainty. The main focus of this paper is to obtain various properties in the controller by changing the degree of homogeneity. The homogeneous controller for perturbed integrator chains is developed from a discontinuous Lyapunov-based controller for pure integrator chains. It is then demonstrated that the homogeneity degree can be controlled in the neighborhood of zero, such that the amplitude of discontinuous control is kept to its minimum possible value when the states have converged. It is also shown that the recently developed ``Fixed-Time" stability notion can be achieved by changing the homogeneity degree. Fixed-time stability  was first introduced in \cite{Andrieu2008}; this term refers to the finite-time stabilization of systems with uniform convergence, i.e. the convergence time is bounded and independent of the system's initial state. In \cite{Zavala_2011} and \cite{Zavala_2012}, uniform convergence to a neighborhood of the origin was demonstrated for second order systems and arbitrary order respectively. In \cite{Polyakov}, fixed-time convergence controllers were developed for linear systems, insuring guaranteed convergence exactly to zero. Based on the control of homogeneity degree, the controller presented in this paper ensures fixed-time convergence to zero of a perturbed chain of integrators.

The paper is organized as follows: the problem formulation as well as the motivation and contributions of the paper are discussed in Section 2. The controller design is presented in Section 3 and its special cases are demonstrated in Section 4. Simulation results are shown in Section 5 and concluding remarks are given in Section 6.

%%%%%%%%%%%%%%%%%%%%%%%%%%%%%%%%%%%%%%%%%%%%%%%%%%%%%%%%%%%%%%%%%%%%%%%%%%%%%%%%
\section{Problem Formulation, motivation and contribution}
The mathematical formulation of the perturbed integrator chain problem is developed first. Then, the motivation behind using homogeneity based controllers and the contribution of this paper are presented.

\subsection{Problem Formulation}
\noindent Let us consider an uncertain nonlinear system:
\beqnum\label{nonlinear}
\left\{
\begin{array}{l}
 \dot x(t) = f(x,t) + g(x,t)u ,\\
 y(t)= s(x,t),
 \end{array}
\right.
\eeqnum
\noindent where $x \in {\mathbb{R}^n}$ is the state vector and $u \in \mathbb{R}$ is the input control. The sliding variable $s$ is a measured smooth output-feedback function and $f(x,t)$ and $g(x,t)$ are uncertain smooth functions. It is assumed that the relative degree, $r$ of the system \cite{Isidori} is globally well defined, uniform and time invariant \cite{Dinuzzo09} and the associated zero dynamics are asymptotically stable. For autonomous systems, $r$ is the minimum order of time derivatives of the output $y(t)$ in which the control input $u$ appears explicitly. This means that, for suitable functions $\tilde \varphi(x,t)$ and $\tilde \gamma(x,t)$, we obtain 

\beqnum\label{nonlinear1}
y^{(r)}(t)=\tilde \varphi(x(t),t)+\tilde \gamma(x(t),t) u(t).
%$ \quad $\tilde\gamma = \frac{\partial}{\partial u}y^{(r)} \neq 0, \\ \tilde\varphi = y^{(r)}|_{u=0}.

\eeqnum
The functions $\tilde \gamma (x,t)$ and $\tilde \varphi (x,t)$ are assumed to be bounded by positive constants  $\gamma_m \leq\gamma_M$ and $\bar \varphi$, such that, for every $x \in {\mathbb{R}^n}$ and $t\geq 0$,

\beqnum\label{bound}
0 < \gamma_m\le \tilde\gamma(x,t) \le \gamma_M, \quad \left|  \tilde\varphi(x,t) \right| \le \bar \varphi.

\eeqnum
Defining $s^{(i)}:= \frac{d^i}{dt^i}y$; the goal of $r^{th}$ order sliding mode control is to arrive at, and keep the following manifold in finite-time:

\beqnum\label{manifold}
s^{(0)}(x,t)=s^{(1)}(x,t)=\cdots=s^{(r-1)}(x,t)=0.

\eeqnum
To be more precise, let us introduce $z=[z_1\  z_2\ ... z_r ]^T:=[s\ \dot s ... \ s^{(r-1)} ]^T$. Then \eqref{manifold} is equivalent to $z=0$. Since the only available information on ${\tilde \varphi(x,t)}$ and ${\tilde \gamma(x,t)}$ are the bounds \eqref{bound}, it is natural to consider a more general control system instead of System \eqref{nonlinear1}, such as

\beqnum\label{u.l.s.}
	\dot z_{i} = z_{i+1}, i=1,\cdots,r-1, \quad \dot z_r = \varphi(t) + \gamma(t) u,

\eeqnum
where the new functions $\varphi$ and $\gamma$ are arbitrary measurable functions that verify the condition

\beqnum\label{H1}
(H1)\quad \varphi(t) \in \left[-\bar\varphi , \bar\varphi \right], \quad \gamma(t) \in \left[\gamma_m , \gamma_M \right].

\eeqnum
%where $\bar\varphi$, and $\gamma_m \leq\gamma_M$ are positive constants.
The objective of this paper is to design controllers which stabilize System \eqref{u.l.s.} to the origin in finite-time . Since these controllers are to be discontinuous feedback laws $u=U(z)$, solutions of  \eqref{u.l.s.} will fall in the realm of  differential inclusions and need to be understood here in Filippov sense, i.e. the right hand vector set is enlarged at the discontinuity points of \eqref{u.l.s.} to the convex hull of the set of velocity vectors obtained by approaching $z$ from all the directions in $\mathbb{R}^r$, while avoiding zero-measure sets \cite{Filippov}.

%%%%%%%%%%%%
\subsection{Motivation behind Homogeneity based control}
The main motivation behind this paper is to develop a Lyapunov-based universal homogeneous controller for an arbitrary order perturbed integrator chain represented by System \eqref{u.l.s.}. The insistence upon homogeneity based control is due to the fact that varying the controller's homogeneity degree produces different interesting results.
%{\color{red}
%We propose a generalization for the power function $a^\alpha$, as

\begin{defn}
The sign function is a  multi-valued function defined on $\mathbb{R}$ by $sign(z)=z/\left| z \right|$ if $z\neq 0$ and $sign(0)=[-1,1]$. Moreover, if  $\alpha \ge 0 $ and $a \in \mathbb{R}$, we use 
 $\left\lfloor a \right\rceil ^\alpha$ to denote $\left| a \right| ^\alpha sign(a)$. Then, for $a,b\in \mathbb{R}$ and $\alpha >0$, it holds 
$\text{sign}(\left\lfloor a \right\rceil^{\alpha} - \left\lfloor b \right\rceil^{\alpha})  = \text{sign}(a-b)$.
\end{defn}
%\begin{proof}
%we take the different cases:
%
%\begin{itemize}
%	\item $\text{sign}(a) = -\text{sign}(b)$, then we should proof that $$\text{sign}(a) \text{sign}\left( |a|^{\alpha} + |b|^{\alpha} \right) = \text{sign}(a) \text{sign}\left( |a| + |b| \right),$$ which is evident.
%	\item $\text{sign}(a) = \text{sign}(b)$, then we should proof that 
%		$$\text{sign}(a) \text{sign}\left( |a|^{\alpha} - |b|^{\alpha} \right) = \text{sign}(a) \text{sign}\left( |a| - |b| \right),$$
%		or 
%		$$\text{sign}\left( |a|^{\alpha} - |b|^{\alpha} \right) = \text{sign}\left( |a| - |b| \right).$$
%		Without restriction of the generality, assume that $|a| \ge |b|$, then $|a|^{\alpha} \ge |b|^{\alpha}$, and the previous equality becomes true.
%\end{itemize}
%This end the proof.
%\end{proof}
%}

Let us present some observations that were made in \cite{Andrieu2008}. Considering the following one-dimensional differential equation

\beqnum\label{eq:diff}
	\dot z = \omega(z) = - c \left\lfloor z \right\rceil^\alpha.
\eeqnum
where $\alpha \ge 0 $, $c>0$.  The degree of homogeneity of Equation \eqref{eq:diff} is $\kappa = \alpha -1$ and this system is stable for all $c>0$ and $ \alpha \ge 0$. However, different characteristics can be obtained in the system, depending upon the value of $\alpha$:
\begin{itemize}
	\item $\alpha = 0$: the convergence to zero occurs in finite-time. The controller $\omega(z)$ is uniformly bounded for $z \in \mathbb{R} $  but discontinuous at $z=0$;
	\item $0 < \alpha < 1$: the convergence to zero occurs in finite-time. The controller $\omega(z)$ is unbounded and tends to zero as $|z| \rightarrow 0$;
	\item $\alpha > 1$: the convergence to zero is asymptotic, however the convergence time to the sphere $\mathbf{B}(0,1)=\{z \in \mathbb{R}:\|z\| < 1\}$ is uniformly bounded by a constant. The controller $\omega(z)$ is unbounded.
	the set $\mathbf{B}(0,\epsilon) = \{z \in {\hat U}:V(z) < \varepsilon\}$ is fixed-time attractive.
\end{itemize}
Let us now consider a perturbed integrator, i.e. \eqref{eq:diff} is replaced by the following differential inclusion:

\beqnum \label{in:diff}
	\dot z \in  \left[-\bar \varphi , +\bar \varphi \right] + u(z) \left[\gamma_m , \gamma_M \right],
\eeqnum
where $\gamma_m \leq\gamma_M$ and $\bar \varphi$ are arbitrary positive constants. The observations mentioned above can be extended to System \eqref{in:diff} by applying a control of the form $u=\frac{1}{\gamma_m}\left ( \omega(z) + \bar \varphi sign( \omega(z))  \right )$ \cite{Harmouche_CDC12}. In addition, manipulation of homogeneity degree $\kappa$ leads us to controllers with the following properties:
\begin{itemize}
	\item fixed-time unbounded controller: switch from $\kappa_1>0$ to $\kappa_2<0$ when $z$ reaches the sphere $\mathbf{B}(0,1)$.
	\item uniformly bounded controller, with reduced amplitude of discontinuous control as $z$ converges to zero: switch from $\kappa_1 = -1 $ to $-1 <\kappa_2< 0$.
\end{itemize}

\subsection{Contribution}
While these observations are well-known, they have so far not been exploited in control algorithms, to the best of our knowledge. Therefore, it is interesting to develop and study algorithms in which the homogeneity degree is varied according to the system's state. 
%{\color{red}
In this work, we extend the above observations related to homogeneous controllers for a single integrator to the stabilization of perturbed integrator chains of arbitrary order $r\geq 1$, based on a controller which stabilizes pure integrator chains. It is shown that, for particular choice of homogeneity degree, a bounded Lyapunov-based arbitrary order controller can be designed, which is similar in structure to Levant's well-known homogeneous controller \cite{Levant2001}. The existence of the Lyapunov function provides the added advantage of analytical tuning of controller parameters. It is also demonstrated that a bounded controller is synthesized using a change of homogeneity degree, such that the controller has a reduced amplitude at $z=0$. Then a fixed-time controller is obtained, also by controlling the homogeneity degree.
%}

%--------------------------------------------------------------
\section{Controller design}
We will now develop the controller in two steps. The stabilization of a pure integrator chain will be considered first. Then the study is extended to the case of a perturbed integrator chain.

\subsection{Useful definitions, lemmas and theorems}
We need the following definitions to state our results. Consider the differential system
%

%\vspace{-0.5cm}
\begin{equation}\label{test}
\dot z =f(t,z), \,\, z\in\mathbb{R}^r.
\end{equation}

\begin{defn}\cite{Polyakov,Bernstein2000}
The equilibrium point $z = 0$ of System \eqref{test} is said to be locally \textbf{finite-time} stable in a neighborhood ${{\hat U}} \subset {\mathbb{R}^r}$ if (i) it is asymptotically stable in ${\hat U}$; (ii) it is finite-time convergent in ${\hat U}$, i.e. for any initial condition $z_0$, $z(t, z_0) = 0 , \forall t \ge T(z_0)$, where $T(z_0)$ is called the settling-time function. The equilibrium point $z = 0$ is globally finite-time stable if ${{\hat U}}= {\mathbb{R}^r}$. The equilibrium point is \textbf{fixed-time} stable if (i) it is globally finite-time stable; (ii) the settling-time function is bounded by a constant $T_{max}$, i.e. $\exists T_{max}>0: \forall z_0 \in {\mathbb{R}^r}, T(z_0) \le T_{max}$.
\end{defn}
\begin{defn}\cite{Polyakov}
The set $S$ is said to be globally \textbf{finite-time} attractive for \eqref{test}, if for any initial condition $z_0$, the trajectory $z(t,z_0)$ of \eqref{test}, achieves $S$ in finite-time $T(z_0)$.
Moreover, the set $S$ is said to be \textbf{fixed-time} attractive for \eqref{test}, if (i) it is globally finite-time stable; (ii) the settling-time function is bounded by a constant $T_{max}$.
\end{defn}

Let us recall the following theorem.

\begin{montheo}\label{cond:bernstein}
\cite{Bernstein2005,Bernstein2000} Suppose there exists a positive definite $C^1$ function $V$ defined on a neighborhood ${{\hat U}} \subset {\mathbb{R}^r}$ of the equilibrium point $z=0$ and real numbers $C>0$ and $\alpha \ge 0,$ such that the following condition is true for every trajectory $z$ of System \eqref{test},

\beqnum\label{eq:bernstein}
	 \dot{V} + C{V}^\alpha(z(t))\leqslant 0, \hbox{ if }z(t) \in {\hat U},
\eeqnum
where $\dot{V}$ is the time derivative of $V(z(t))$. (Here for $\alpha = 0$, Equation \eqref{eq:bernstein} means $\dot V \le -C$ if $z(t) \in \hat U \setminus \{0\}$.) Then all trajectories of System \eqref{test} which stay in ${{\hat U}}$ converge to zero.
If ${{\hat U}} = {\mathbb{R}^r}$ and $V$ is radially unbounded, then System \eqref{test} is globally stable with respect to the equilibrium point $z=0$.\\
Depending on the value $\alpha$, we have different types of convergence: if $0\le \alpha < 1$,  the equilibrium point $z=0$ is finite-time stable (\cite{Bernstein2000}), if $\alpha = 1$, it is exponentially stable
and if $\alpha > 1$ the equilibrium point $z=0$ is asymptotically stable equilibrium and, for every $\epsilon>0$, the set $\mathbf{B}(0,\epsilon) = \{z \in {\hat U}:V(z) < \varepsilon\}$ is fixed-time attractive.
\end{montheo}
%%%%%%%%%%%%%%%%%%%%%%%%%%%%%%%%%%%%%%%%%%%%%%%%%%%%%%%%%%%%%%%%%%%
\begin{proof} The argument is obvious (cf.  \cite{Bernstein2000} for $0 <\alpha< 1$). Let us just note that for $\alpha > 1$ with initial condition $V(z(0)) = V_0$, an integration of $ \dot{V} + C{V}^\alpha\leq 0$ shows that every trajectory enters the neighborhood defined by $V(z) \le \varepsilon$ in a fixed time less than or equal to $\frac1{C(\alpha - 1)\varepsilon^{\alpha - 1}}$ for any initial condition.
\end{proof}
%%%%%%%%%%%%%%%%%%%%%%%%%%%%%%%%%%%%%%%%%%%%%%%%%%%%%%%%%%%%%%%%%%%

We next recall the concept of homogeneity. %Consider the time-invariant differential system
%\vspace{-0.5cm}
%\begin{equation}\label{test1}
%\dot z =f(z), \,\, z\in\mathbb{R}^r.
%\vspace{-0.5cm}
%\end{equation}
%%%
\begin{defn}\label{dil}\cite{Hong}
The family of dilations $\zeta_\epsilon^p$, $\epsilon>0$, are the linear maps defined on $\mathbb{R}^r$
given by
$$
	\zeta_\epsilon^p(z_1,\cdots,z_r) = (\epsilon^{p_1}z_1,\cdots,\epsilon^{p_r}z_r) ,	
$$
where $p=(p_1,\cdots,p_r)$ with the dilation coefficients $p_i>0$, for $i=1,\cdots,r$.
\end{defn}
%%%
%\begin{defn}\cite{Hong}
%The vector field $f(z) = \left( f_1(z),\cdots,f_r(z) \right)^T$ is homogeneous of degree $\kappa \in \mathbb{R}$ with respect to the family of dilation $\zeta_\epsilon^p$ if, for every $z\in\mathbb{R}^r$ and $\epsilon>0$,\\
%$
	%f_i(\epsilon^{p_1}z_1,\cdots,\epsilon^{p_r}z_r) = \epsilon^{p_i + \kappa} f_i(z_1,\cdots,z_r),
	%\  i=1,\cdots,r, \  \epsilon >0.\\
%$
%System \eqref{test1} is called homogeneous, if the vector field $f(z)$ is homogeneous.
%\end{defn}
\begin{defn}\cite{Zhang2013} and \cite{Hong}
A differential inclusion $\dot z \in F(z)$
is said to be a homogeneous differential inclusion of degree $\kappa \in \mathbb{R}$ with respect to the family of dilation $\zeta_\epsilon^p$ if it satisfies
\beq
	F_i(\epsilon^{p_1}z_1,\cdots,\epsilon^{p_r}z_r) = \epsilon^{p_i + \kappa} F_i(z_1,\cdots,z_r),\  i=1,\cdots,r, \  \epsilon >0.
\eeq
where $F(z): \mathbb{R}^r \rightarrow \mathbb{R}^r$, $F(z) = \left( F_1(z),\cdots,F_r(z) \right)^T$ and $\zeta_\epsilon^p$ is given in Definition \ref{dil}.

A function $\Omega(z)$ is homogeneous of degree $a > 0$ with respect to the family of dilation $\zeta_\epsilon^p$ where $a$ is a positive real number if, for every $z\in\mathbb{R}^r$ and $\epsilon>0$, $\Omega(\epsilon^{p_1}z_1,\cdots,\epsilon^{p_r}z_r) = \epsilon^a \Omega(z_1,\cdots,z_r)$.
\end{defn}

%{\color{red}
\begin{defn}
Let $p_j$, $j=1,\cdots,i$, $\kappa$ and $c$ given as follows
\beq
	p_j = 1+ (j-1) \kappa, \quad \kappa \in [-1/r,1/r],\ c \ge \max(p_1,\cdots,p_r)>0.
\eeq
The homogeneous norm $\Gamma_{i,c}(z)$ for $z \in \mathbb{R} ^ i $ is defined by
\beq
	 \Gamma_{i,c}(z) \equiv \Gamma_{i,c}(z_1,\cdots,z_i) = \left(\sum\limits_{j=1}^{i} {\left| z_j \right|^{c/p_j} } \right)^{1 /c},
\eeq
where $\Gamma_{i,c}(\epsilon^{p_1}z_1,\cdots,\epsilon^{p_i}z_i) = \epsilon\Gamma_{i,c}(z_1,\cdots, z_i)$, $\epsilon >0 $. In this case, the unit sphere $S_{i,c}$ is given by  $S_{i,c}= \{ z \in \mathbb{R} ^ i :  \Gamma_{i,c}(z)  = 1  \}$.
\end{defn}
%}
The following lemmas are used in the course of some subsequent arguments.
\begin{lemme}\label{lem1}
For $\alpha\geq 1$, define on $\mathbb{R}^2$ the functions $w(b,a)= \left\lfloor  b\right\rceil^{\alpha} - \left\lfloor  a \right\rceil^{\alpha}$
and $W(b,a)=\int_a^b w(s,a)ds$. Then the function $g(b,a)=\frac{W(b,a)}{\vert w(b,a)\vert^{\frac{\alpha+1}{\alpha}}}$ is continuous on 
$\mathbb{R}^2\setminus \{(0,0\}$ and homogeneous of degree zero with respect to $\zeta_\epsilon^{(1,1)}$. In particular, for every $\beta>0$, the function $g^{\beta}$ is uniformly bounded over $\mathbb{R}^2\setminus \{(0,0\}$.
\end{lemme}
\begin{proof} The only fact non trivial to establish is that $g$ is well-defined on $\mathbb{R}^2\setminus \{(0,0\}$. 
One can assume with no loss of generality that both $a$ and $b$ are positive real numbers. Set $B:=b/a$. Then $g(b,a)=g(B,1)$ and we are left to prove that $g(B,1)$ is continuous at $B=1$. It is immediate to see that the latter fact hold true by taking the Taylor's expansions of both $W(B,1)$ and $w(B,1)$ in a neighborhood of $B=1$. 
\end{proof}
%
%
%
%
%
%\begin{lemme}[Lemma 4.2 of \cite{Bernstein2005}]\label{lem1}
%Suppose $\Omega_1$ and $\Omega_2$ are continuous real-valued functions on $\mathbb{R}^r$, homogeneous with respect to $\zeta_\epsilon^p$ of degrees $d_1>0$ and $d_2>0$, respectively, and $\Omega_1$ is positive definite. Then, for every $z \in \mathbb{R}^r$, 
%$$
%		\left[ \min_{\{z:\Omega_1(z)=1 \}} \Omega_2(z)\right]	[\Omega_1(z)]^{{d_2}/{d_1}} \le  \Omega_2(z)   \le \left[ \max_{\{z:\Omega_1(z)=1 \}} \Omega_2(z) \right]  [\Omega_1(z)]^{{d_2}/{d_1}}.
%$$
%\end{lemme}
\begin{lemme}\label{theta0}
For every $\theta>0$, positive integer $i$ and non negative real numbers $a_1,\cdots,a_i$, one has that 
$(\sum_{j=1}^i a_j)^{\theta}\leq \max(1,i^{\theta-1})\sum_{j=1}^i a_j^{\theta}$.
\end{lemme}
\begin{proof}
The result is immediate for either $i=1$ or $\theta\geq 1$, since it follows from the convexity of the function $x\mapsto x^{\theta}$ defined on 
$\mathbb{R}_+$. Assume now that $i>1$ and $\theta< 1$. Let $\Delta_i:=\{a=(a_1,\cdots,a_i)\in (\mathbb{R}_+)^i,\ \sum_{j=1}^i a_j=1\}$,  $f_i$ be the real-valued function given by $f_i(a)=\sum_{j=1}^i a_j^{\theta}$ and $C_i$ be the minimum of $f_i$ over $\Delta_i$, which is well defined since $\Delta_i$ is compact and $f_i$ is continuous. It is clear that $C_{i+1}\leq C_i\leq C_1=1$. If $C_i$ is reached at an interior point $\bar a$ of $\Delta_i$, then a trivial application Lagrange's theorem shows that all the coordinates of $\bar a $ are equal to $1/i$ implying that $C_i=i^{1-\theta}>1$, which is not possible. Then, $C_i$ is reached at a boundary point $\bar a$ of $\Delta_i$ and a trivial induction yields the result.
\end{proof}

\begin{Proposition}[\cite{Rosier1992}]\label{prop1}
Let $\Omega$ be a positive definite $C^1$ function, homogeneous of degree $a$ with respect to $\zeta_\epsilon^p$.
Then, for all $i = {1,\cdots, r}$; ${\partial \Omega}/{\partial z_i}$ is homogeneous of degree $(a - p_i)$.
\end{Proposition}
%%%%%%%%%%%%%%%%%%%%%%%%%%%%%%%%%%%%%%%%%%%%%%%%%%%%%%%%%%%%%%%%%%%%%%%%%%%%%%%%%%%%%%%%%%%%%%%%%%%%%%%%
%%%%%%%%%%%%%%%%%%%%%%%%%%%%%%%%%%%%%%%%%%%%%%%%%%%%%%%%%%%%%%%%%%%%%%%%%%%%%%%%%%%%%%%%%%%%%%%%%%%%%%%%

\subsection{Stabilization of a pure integrator chain}
Consider the following pure integrator chain:

\beqnum\label{pure_integrators}
%\left\{
%\begin{array}{lcl}
 \dot z_i = z_{i+1}, \ i=1,...,r-1,\quad
 \dot z_{r} = u .
 %\end{array}
 %\right.

\eeqnum
The following Hong's controller guarantees the stabilization of \eqref{pure_integrators}.
\begin{montheo} \cite{Hong} \label{Levant_theorem}
Let $r$ be the order of the pure integrator chain given in \eqref{pure_integrators}. For $\kappa \in [-1/r,1/r]$, set
$
	p_i= 1 + (i-1)\kappa, \quad i=1,\cdots, r.
$
Then there exist constants $l_i>0,\ i=1,\cdots,r$, independent of $\kappa$, such that the feedback control law $u = \omega_\kappa^{H}(z):= v_r$ defined inductively by:
\beqnum%\label{Levant_control}
 v_0 = 0,\
v_{i} = -l_{i} \lfloor\lfloor z_{i} \rceil^{\beta_{i-1} } - \lfloor v_{i-1} \rceil^{\beta_{i-1} } \rceil^{(p_i + \kappa) /( p_i\beta_{i-1})},

\eeqnum
stabilizes System \eqref{pure_integrators},
where $\beta_i$ are defined by
\beqnum\label{f_N}
	\beta_0 =p_2,&&(\beta_i + 1)p_{i+1} = \beta_0 + 1 > 0,\,\quad i=1,...,r-1.\\
\eeqnum
There also exists a homogeneous Lyapunov function $V_{\kappa,r}(z)$ for the closed-loop system \eqref{pure_integrators} with the state-feedback $u$, that satisfies $\dot V_{\kappa,r} \le -C V_{\kappa,r}^{{(2+2\kappa)}/{(2 + \kappa)}}$, for some positive constant $C$, independent of $\kappa$.
\end{montheo}
\begin{rem}
The remarkable feature of the above result lies in the explicit construction of both the controller and the Lyapunov function that we recall next. For $1\leq i\leq r$, define\beq
		w_i(z_1,\cdots,z_i) &=& \left\lfloor z_i \right\rceil ^{\beta _{i - 1}} - \left\lfloor {v_{i - 1}} \right\rceil ^{\beta _{i - 1}},\\
		W_i(z_1,\cdots,z_i) &=& \int_{{v_{i - 1}}({z_1},...,{z_{i - 1}})}^{{z_i}} w_i\left( z_1,...,z_{i - 1},s \right) ds,\\
		 &=& \frac{1}{\beta _{i - 1} + 1} \left( \left| z_i \right|^{\beta _{i - 1} + 1} -\left| v_{i - 1} \right|^{\beta _{i - 1} + 1}  \right) - \left\lfloor v_{i - 1} \right\rceil^{\beta _{i - 1}} \left( z_i - v_{i - 1} \right).
\eeq
Then the Lyapunov function $V_{\kappa,r}$ is defined by $V_{\kappa,r}(z) = \sum_{i=1}^{r}{W_i}(z_1,\cdots,z_i)$.
\end{rem}

We present next a modified version of Hong's controller denoted $\omega_\kappa^{HM}$ which also guarantees the stabilization of \eqref{pure_integrators}. 

%%%%%%%%%%%%%%%%%%%%%%%%%%%%%%%%%%%%%%%%%%%%%%%%%%%%%%%%%%%%%%%%%%%%%%%%%%%%%%%%%%%%%%%%%%%%%%%%%%%%%%%%%
\begin{montheo} \label{Hong_modif}
Let $r$ be the order of the pure integrator chain given in \eqref{pure_integrators}. For $\kappa \in [-1/r,0]$, set
$
	p_i= 1 + (i-1)\kappa, \quad i=1,\cdots r,
$
%\eeq
and let $c$ be a positive constant such that $c \ge \max(p_1,\cdots,p_r)$.
Then there exist constants $l_i>0,\ i=1,\cdots,r$, independent of $\kappa$, such that the feedback control law $u = \omega_\kappa^{HM}(z):= v_r$ defined by: 
\beqnum%\label{Levant_control}
%\left\{
%\begin{array}{lcl}
 v_0 = 0,\
 v_i = -l_i N_i ,\
 			 i=1,\cdots,r,
 %\end{array}
 %\right.
%\vspace{-0.5cm}
\eeqnum
stabilizes System \eqref{pure_integrators},
where $N_i=\left\lfloor { \left\lfloor  z_i \right\rceil^{c/p_i} - \left\lfloor v_{i-1} \right\rceil ^{c/p_i} }\right\rceil ^{{(p_{i}+\kappa)}/{c}}$. There also exists a homogeneous Lyapunov function $V_{\kappa,r}(z)$ for the closed-loop system \eqref{pure_integrators} under $u$, that satisfies $\dot V_{\kappa,r} \le -C V_{\kappa,r}^{{(c+1+\kappa)}/{(c+1)}}$, for some positive constant $C$ independent of $\kappa$.
\end{montheo}
\begin{proof} The argument largely follows the lines of  \cite{Hong}. 
For $1\leq i\leq r$, we define
%\vspace{-0.3cm}
\beqnum\label{f_W}
	%%a &=& 1 + p_2 = 2+k \quad \ge 1,\\
	%w_i &:=& \left(\sum\limits_{j=1}^{i} {\left| z_j \right|^{c/p_j} } \right) sign\left( z_i - v_{i-1}  \right),\\
	w_i(z_1,\cdots,z_i)&:=& \left\lfloor  z_i \right\rceil^{c/p_i} - \left\lfloor  v_{i-1} \right\rceil^{c/p_i} ,\\
	%W_i &:=& \int_{{v_{i - 1}}({z_1},...,{z_{i - 1}})}^{{z_i}} w_i\left( z_1,...,z_{i - 1},s \right) ds,\\
			%&=& \left( |z_1|^{\frac{c}{p_1}} + \cdots + |z_{i-1}|^{\frac{c}{p_{i-1}}} \right)  \left| z_i - v_{i-1} \right|  + {\left|\left\lfloor z_i \right\rceil^{\frac{c}{p_i} + 1}  -   \left\lfloor v_{i-1} \right\rceil^{\frac{c}{p_i} + 1} \right| } / {\left(1 + {c}/{p_i} \right)}.
W_i(z_1,\cdots,z_i) &=& \int_{v_{i-1}}^{z_i}{  \left\lfloor s \right\rceil^{\frac{c}{p_i}}  -   \left\lfloor v_{i-1} \right\rceil^{\frac{c}{p_i}}      ds        },\\
				&=&	\frac{ \left| z_i  \right|^{\frac{c}{p_i} + 1} - \left| v_{i-1}  \right|^{\frac{c}{p_i} + 1}}{\frac{c}{p_i} + 1} - \left\lfloor v_{i-1}\right\rceil ^{\frac{c}{p_i}} \left(  z_i - v_{i-1} \right).
\eeqnum
It can be seen that $W_i$ is positive definite function with respect to $v_{i-1}-z_i$, homogeneous with respect to $ \xi_p^\epsilon $ of degree $(c+p_i)$. We introduce $\bar W_i := W_i^{\delta_i}$, where $\delta_i = {(c+1)}/{(c+p_i)}$, so that all functions $\bar W_i$ are homogeneous of the same homogeneity degree $(c+1)$. 
%We get at once, for $1\leq i\leq r$, that there exist positive constants $k_i$, such that for every $z\in\mathbb{R}^r$, $\bar W_i(z_1,\cdots,z_i) \le k_i \left| w_i(z_1,\cdots,z_i) \right|^{{(c + 1)}/{c} }$.
	%\begin{proof}We can get that $\left(\sum\limits_{j=1}^{i} {\left| z_j \right|^{c/p_j} } \right)$ is a homogeneous function with respect to $ \xi_p^\epsilon $ of degree $c$. Then according to Lemma \ref{lem1}, and for a given $\kappa$, there exists a constant $K_i$ depending on $\kappa$, such that
%$
	%\bar W_i \le K_i(\kappa) \left(\sum\limits_{j=1}^{i} {\left| z_j \right|^{c/p_j} } \right),
%$
%where $K_i(\kappa) = \max_{z\in S_i} \bar W_i $. Then, the choice of $k_i$, as
%$
	%k_i = \max_{\kappa \in[ -1/r,1/r]} K_i(\kappa),
%$
%implies \eqref{ineq:ki}.
%\end{proof}
%%%%%
We proceed to prove the theorem by induction on $r$.\\
\textit{Step 1:} Consider $\dot z_1 = u$. For any $l_1 > 0$, taking
$u = \omega_\kappa^{HM}(z_1) = -l_1 \left\lfloor  z_1 \right\rceil^{{(p_1 + \kappa)} / p_1}$ stabilizes the closed-loop system.
The Lyapunov function $V_{\kappa,1} = W_1=\left| z_1 \right|^{1 + c}/(1+c)$ is homogeneous of degree $c+1$ and
$
\dot V_{\kappa,1} = -l_1 \left| z_1 \right|^{ c + p_2}  \le  -\eta_1  V_{\kappa,1} ^{{ (c + 1 + \kappa )} / { (c+1 )}},\
$
for some constant $\eta_1>0$.\\
\textit{Step $i$:}
Assume that the conclusion holds true till $i-1$. Define the Lyapunov function $V_{\kappa,i}$ by	$V_{\kappa,i} = V_{\kappa,i-1} + \bar W_i = \sum\limits_{j = 1}^{i } {\bar W_j}$. Then, 
\beqnum \label{lyap_hm}
	\begin{array}{lll}
		\dot V_{\kappa,i} &=& \sum\limits_{j = 1}^{i - 1}  \!\!\!{\frac{{\partial {\bar W_i}}}{{\partial {z_j}}}{z_{j + 1}}}  \!+\! w_i v_i W_i^{\frac{-(i-1)\kappa}{c+p_i}}
		+ {{\dot { V}}_{\kappa, i - 1}}  +  \frac{{\partial { V_{\kappa, i - 1}}}}{{\partial {z_{i - 1}}}}\left( {{z_i}  -  {v_{i - 1}}} \right) ,\\
		 &=&  \sum\limits_{j = 1}^{i - 1} {\frac{{\partial {\bar W_i}}}{{\partial {z_j}}}{z_{j + 1}}} -  l_i |w_i|^{\frac{c+p_{i} + \kappa}{c}} \bar W_i^{\frac{-\kappa(i-1)}{c+p_{1}}}
		 + {{\dot { V}}_{\kappa, i - 1}}  +  \frac{{\partial { V_{\kappa, i - 1}}}}{{\partial {z_{i - 1}}}}\left( {{z_i}  -  {v_{i - 1}}} \right).
\end{array}	\eeqnum	
Using Lemma~\ref{lem1}, one gets firstly that  there exists $k_i>0$ such that for every non zero $(z_1,\cdots,z_i)$, one has $\bar W_i(z_1,\cdots,z_i)/ \left| w_i(z_1,\cdots,z_i) \right|^{{(c + 1)}/{c} }\le k_i $ and secondly
$$
-  l_i |w_i|^{\frac{c+p_{i} + \kappa}{c}} \bar W_i^{\frac{-\kappa(i-1)}{c+p_{1}}}
\leq  - l_i \frac{\bar W_i^{\frac{c+p_{i} + \kappa}{c+p_{1}}}}{k_i^{\frac{c +1 + \kappa}{c+1}}}.
$$
%
%
%
%
%
%
%\beqnum \label{lyap_hm}
%	\begin{array}{lll}
%		\dot V_{\kappa,i} &=& \sum\limits_{j = 1}^{i - 1}  \!\!\!{\frac{{\partial {\bar W_i}}}{{\partial {z_j}}}{z_{j + 1}}}  \!+\! w_i v_i W_i^{\frac{-(i-1)\kappa}{c+p_i}}
%		+ {{\dot { V}}_{\kappa, i - 1}}  +  \frac{{\partial { V_{\kappa, i - 1}}}}{{\partial {z_{i - 1}}}}\left( {{z_i}  -  {v_{i - 1}}} \right) ,\\
%		 &=&  \sum\limits_{j = 1}^{i - 1} {\frac{{\partial {\bar W_i}}}{{\partial {z_j}}}{z_{j + 1}}} -  l_i |w_i|^{\frac{c+p_{i} + \kappa}{c}} \bar W_i^{\frac{-\kappa(i-1)}{c+p_{1}}}
%		 + {{\dot { V}}_{\kappa, i - 1}}  +  \frac{{\partial { V_{\kappa, i - 1}}}}{{\partial {z_{i - 1}}}}\left( {{z_i}  -  {v_{i - 1}}} \right) ,\\
%		 &\le&  \sum\limits_{j = 1}^{i - 1} {\frac{{\partial {\bar W_i}}}{{\partial {z_j}}}{z_{j + 1}}}   - l_i \frac{\bar W_i^{\frac{c+p_{i} + \kappa}{c+p_{1}}}}{k_i^{\frac{c + p_i + \kappa}{c+p_1}}} \bar W_i^{\frac{-\kappa(i-1)}{c+p_{1}}}
%		 + {{\dot { V}}_{\kappa, i - 1}} + \frac{{\partial { V_{\kappa, i - 1}}}}{{\partial {z_{i - 1}}}}\left( {{z_i}  -  {v_{i - 1}}} \right) ,\\
%		 &\le&  \sum\limits_{j = 1}^{i - 1}   {\frac{{\partial {\bar W_i}}}{{\partial {z_j}}}{z_{j + 1}}}  - \frac{l_i}{k_i^{\frac{c + p_i + \kappa}{c+p_1}}} \bar W_i^{\frac{c + p_2}{c+p_{1}}}
%		 + {{\dot { V}}_{\kappa, i - 1}} + \frac{{\partial { V_{\kappa, i - 1}}}}{{\partial {z_{i - 1}}}}\left( {{z_i}  - {v_{i - 1}}} \right) .\\
%		\vspace{-0.8cm}
%\end{array}
%\eeqnum
The fact that $\bar W_i$ are homogeneous with respect to $\zeta_\epsilon^p$ of degree $(c+1)$ for each $i=1,\cdots,r$, implies that $V_i$ are homogeneous of degree $(c+1)$ with respect to $\zeta_\epsilon^p$ as well. In addition, according to Proposition $\ref{prop1}$, $\dot V_i$ are homogeneous of degree $(c+1- \kappa )$ with respect to $\zeta_\epsilon^p$. Then without loss of generality, the study can be restricted to the unit sphere $ S_{i,c} $. Set
$$
V^0_i(z_1,\cdots,z_i):= \sum_{j=1}^{i-1}{\frac{{\partial {\bar W_i}}}{{\partial {z_j}}}{z_{j + 1}}}+ \frac{{\partial { V_{\kappa, i - 1}}}}{{\partial {z_{i - 1}}}}\left( {{z_i}  - {v_{i - 1}}} \right),
$$
and define $S^+_i=\{z\in S_{i,c}\ \ V^0_i(z_1,\cdots,z_i)\geq 0\}$. The key point is that
$\min_{z\in S^+_i,\kappa  \in  [-1/r , 0]} \bar W_i^{\frac{c + p_2}{c+p_{1}}}>0$. One can then choose 
$l_i>0$ independent of $\kappa  \in  [-1/r , 0]$ such that, by setting $\eta_i:=l_i/2k_i^{\frac{c + p_i + \kappa}{c+1}}$, we get
$
	\dot V_i  \le  - \sum\limits_{j = 1}^{i} {\eta_j { \bar W}_j^ {\frac{c+1+ \kappa }{c+1}}}.
$ 
At the final step, all parameters $l_i$ are determined, by
$V_{\kappa,r}(z) = \sum\limits_{j = 1}^r{\bar W_j}$ and
$
\dot V _{\kappa,r}(z) \le - \sum\limits_{j = 1}^{i} { \eta_j \bar W_j^ {\frac{c+1+\kappa}{c + 1}}} \le  - \eta \sum\limits_{j = 1}^{i} {  \bar W_j^ {\frac{c+1+\kappa}{c + 1}}},
$
where $\eta: = \min_{1\leq i\leq r}{\eta_i}$. Using Lemma~\ref{theta0}, one gets that
%\beq\label{dd}
	%\sum\limits_{j = 1}^{i} {  \bar W_j^ {\frac{c+1+\kappa}{c + 1}}} \geq  \mu \left (\sum \limits_{j = 1}^{i}   \bar W_j \right )^ {\frac{c+1+\kappa}{c + 1}},
	%\quad \mu: = \min_{1\leq i\leq r,\ \kappa \in [-1/r,+1/r]} (1,i^{-\frac{\kappa}{c+1}}).
%\eeq

\beq\label{dd}
	\sum\limits_{j = 1}^{i} {  \bar W_j^ {\frac{c+1+\kappa}{c + 1}}} \geq  \left (\sum \limits_{j = 1}^{i}   \bar W_j \right )^ {\frac{c+1+\kappa}{c + 1}}.
	%\quad \mu: = \min_{1\leq i\leq r,\ \kappa \in [-1/r,+1/r]} (1,i^{-\frac{\kappa}{c+1}}).
\eeq
Finally we get $\dot V_{\kappa,r} \le -\eta V_{\kappa,r}^{{(c+1+ \kappa )} / {(c + 1)}}$.%, with $C \ge \eta \mu$.
\end{proof}

\begin{rem}
The controller $\omega^{HM}_r$ is only defined for $\kappa \in [\frac{-1}{r}, 0]$ as one can see from Eq.~\eqref{lyap_hm}.

\end{rem}
%%%%%%%%%%%%%%%%%%%%%%%%%%%%%%%%%%%%%%%%%%%%%%%%%%%%%%%%%%%%%%%%%%%%%%%%%%%%%%%%%%%%%%%%%%%%%%%%%%%%%%%%
%{\color{red}
%\begin{rem}
%In practice, the parameters $l_i$ presented in Equation \eqref{li} can be found as follows:
%\begin{itemize}
%	\item Take $N$ values of $\kappa$ : ($\kappa_1,\ \kappa_2,\cdots, \kappa_N$).
%	\item Solve the mathematical optimization problem $l_{i,m}(\kappa_j) = \mathop {\max }\limits_{z \in {S_{i,c}}} \left\{  { \sum\limits_{j = 1}^{i - 1} {\frac{{\partial {\bar W_i}}}{{\partial {z_j}}}{z_{i + 1}}} + \frac{{\partial {V_{i - 1}}}}{{\partial {z_{i - 1}}}}\left( {{z_i} - {v_{i - 1}}} \right)  }  \right\}$ for all $\kappa_j$, $j=1,\cdots, N$.
%	\item Finally, choose $l_i$ greater than the maximum of all $l_{i,m}(\kappa_j)$.
%\end{itemize}
%\end{rem}
%}
%%%%%%%%%%%%%%%%%%%%%%%%%%%%%%%%%%%%%%%%%%%%%%%%%%%%%%%%%%%%%%%%%%%%%%%%%%%%%%%%%%%%%%%%%%%%%%%%%%%%%%%%

\subsection{Stabilization of an $r$-perturbed chain of integrator}
From the controllers $\omega_\kappa^{H}(z)$ and $\omega_\kappa^{HM}(z)$ obtained in Theorem \ref{Levant_theorem} and Theorem \ref{Hong_modif}, we now proceed to the stabilization of the perturbed integrator chain presented in System \eqref{u.l.s.}. The extension of Theorem \ref{Levant_theorem} to the case of System \eqref{u.l.s.} is based on the following result.
%%%%%%%%%%%%%%%%%%
\begin{montheo}\label{Harmouche_condition}\noindent Let $\omega(z)$ and $V(z)$ be respectively, a state-feedback control law stabilizing System \eqref{pure_integrators} and a Lyapunov function for the closed-loop system, which satisfy Theorem \ref{cond:bernstein} and obey the following additional conditions: for every $z\in {\hat U}$,

\beqnum\label{condition_theorem2}
 \frac{\partial V }{\partial z_r}(z) \omega (z)  \leq 0\hbox{ and }
\omega (z) = 0  \Rightarrow   \frac{\partial {V} }{\partial {z_r}} (z) = 0.
\eeqnum
Then, for arbitrary constants $P,Q\ge 1$, the following control law stabilizes System \eqref{u.l.s.}:

\beqnum \label{transf}
	u(z)  =P (\omega (z)  + Q \bar\varphi sign(\omega (z) ))/ {\gamma_m}.
\eeqnum
The function $V (z)$ remains a Lyapunov function for the closed-loop system and satisfies Condition \eqref{eq:bernstein}. If ${{\hat U}} = {\mathbb{R}^r}$ and $V(z)$ is radially unbounded, then the closed-loop system is globally stable with respect to the origin.
\end{montheo}
\begin{proof}
This theorem is a generalization of Theorem 2 of \cite{Harmouche_CDC12}, where it has been proven for $P=Q=1$, and is established in the same way.
\end{proof}

\begin{rem}
The controllers $\omega_\kappa^{H}(z)$ and $\omega_\kappa^{HM}(z)$ satisfy the geometric condition \eqref{condition_theorem2} imposed in Theorem \ref{Harmouche_condition}. Indeed, one gets for $z\in \mathbb{R}^r$, 
\beq
\frac{\partial V_{\kappa,r}^H}{\partial z_r} \ \omega_\kappa^H = -l_r\vert z_i ^{\beta _{i - 1}} - v_{i - 1}^{\beta _{i - 1}}\vert^{2(1+\kappa)/p_i\beta_{i-1}},\quad
\frac{\partial V_{\kappa,r}^{HM}}{\partial z_r} \ \omega_\kappa^{HM} =-l_r\vert  z_j^{c/p_j} -v_{i-1}^{c/p_j} \vert^{1+(p_i+\kappa)/c}.
%	
%	
%	
%	
%	
%	
%	\frac{\partial V_{\kappa,r}}{\partial z_r} \ \omega_\kappa^H = \left( \left\lfloor z_i \right\rceil ^{\beta _{i - 1}} - \left\lfloor {v_{i - 1}} \right\rceil ^{\beta _{i - 1}} \right) \times -l_r \left\lfloor \left\lfloor z_i \right\rceil ^{\beta _{i - 1}} - \left\lfloor {v_{i - 1}} \right\rceil ^{\beta _{i - 1}}\right\rceil ^{(p_i + \kappa) /( p_i\beta_{i-1})} \le 0,
%\eeq
%and
%%
%\beq
%	\frac{\partial V_{\kappa,r}}^{\partial z_r} \ \omega_\kappa^{HM} =-l_r\vert  z_j \right\rceil^{c/p_j} - \left\lfloor v_{i-1} \right\rceil ^{c/p_j} \vert^{2(1+\kappa)/p_i\beta_{i-1}}, 
%	
%	
%	
%	\left( { \left\lfloor  z_j \right\rceil^{c/p_j} - \left\lfloor v_{i-1} \right\rceil ^{c/p_j} }\right)
%	\times -l_r \left\lfloor { \left\lfloor  z_j \right\rceil^{c/p_j} - \left\lfloor v_{i-1} \right\rceil ^{c/p_j} }\right\rceil ^{{(p_{i}+\kappa)}/{c}} \le 0,
\eeq
%
%
%$\omega_\kappa = 0 \Rightarrow -l_r \left| w_r \right|^{{(p_{r}+\kappa)}/{c}} sign(z_r - v_{r-1}) = 0 \Rightarrow 	{\partial V_\kappa} / {\partial z_r} \equiv  {\partial \bar W_r} / {\partial z_r} = 0.$
\end{rem}

\begin{rem}
The controller $u(z)$ presented in Equation \eqref{transf} is clearly discontinuous. However, its absolute value $\left| u(z) \right|$ is equal to 
$P( \left| \omega (z) \right| + Q\bar\varphi)\gamma_m$. Then $\lim_{\|z\| \to 0} \left| u(z) \right|$ takes its minimal value at the origin if $\omega (z)$ vanishes there. In particular when $P=Q=1$, it has been claimed  in Section 2 of \cite{Levant2005} that in order to stabilize the uncertain System \eqref{u.l.s.} by a state-feedback controller $u = u(z)$, it is necessary that the controller be discontinuous at $z=0$ and 
$
	\lim_{\|z\| \to 0} \left |u(z) \right| \ge {\bar \varphi} / {\gamma_m}=:M_{min}.
$
\end{rem}

%%%%%%%%%%%%%%%%%%%%%%%%%%%%%%%%%%%%%%%%%%%%%%%%%%%%%%%%%%%%%%%

\section{Discussion of Special Cases}
Let us now consider some results that arise for some specific choices of the homogeneity degree. First,  a bounded controller with minimum amplitude $M_{min}$ of discontinuous control at $z=0$ is designed. Finally, a controller with fixed-time convergence is synthesized.%
%
%
%it is shown that for a particular choice of homogeneity degree, we obtain a bounded homogeneous controller that is similar to the one presented in \cite{Levant2001}. Then, a bounded controller with minimum amplitude $M_{min}$ of discontinuous control at $z=0$ is designed. Finally, a controller with fixed-time convergence is synthesized.
%

%%%%%%%%%%%%%%%%%%%%%%%%%%%%%%%%%%%%%%%%%%%%%%%%%%%%%%%%%%%%%%%%%%%%%%%%%%%%%%%%%%%%%%%%%%%%%%%%%
%{\color{red}
%\subsection{Connection with Levant's Controller}
%Consider the following bounded controllers.
%\begin{Proposition}
%Let $\omega(z)$, one of the controllers presented in Theorem \ref{Levant_theorem}, or Theorem \ref{Hong_modif}. Then for $\kappa = \frac{-1}{r}$, the controller $\omega(z)$ is bounded, and the controller $u$ defined as
%\beqnum
%	 u = \frac{1}{\gamma_m}\left( \omega(z) + \bar \varphi sign(\omega(z)) \right) \equiv \frac{l_r + \bar \varphi}{\gamma_m} \text{sign}(\omega(z)),
%\eeqnum
%stabilizes System \eqref{u.l.s.} in finite time.
%\end{Proposition}
%
%\begin{rem} It can be noted that the controller presented in this proposition is identical to that of Levant \cite{Levant2001} for $r=1$ and $r=2$. The advantage of our controller in these cases is that the Lyapunov function provides an analytical method of tuning the controller parameters, whereas the tuning is empiric in Levant's case.\end{rem}
%
%}
%%%%%%%%%%%%%%%%%%%%%%%%%%%%%%%%%%%%%%%%%%%%%%%%%%%%%%%%%%%%%%%%%%%%%%%%%%%%%%%%%%%%%%%%%%%%%%%%%
%%%%%%%%%%%%%%%%%%%%%%%%%%%%%%%%%%%%%%%%%%%%%%%%%%%%%%%%%%%%%%%%%%%%%%%%%%%%%%%%%%%%%%%%%%%%%%%%%
\subsection{Homogeneous controller with minimum amplitude of discontinuous control at $z=0$}
We first notice that, for $\kappa =-1/r$, if  $\omega(z)$ denotes one of the controllers presented in Theorem \ref{Levant_theorem} or Theorem \ref{Hong_modif}, then $\omega(z)$ is bounded and the corresponding controller $u$ defined as
\beqnum\label{toto0}
	 u = \frac{1}{\gamma_m}\left( \omega(z) + \bar \varphi sign(\omega(z)) \right) \equiv \frac{l_r + \bar \varphi}{\gamma_m} \text{sign}(\omega(z)),
\eeqnum
stabilizes System \eqref{u.l.s.} in finite time. Moreover, the above controller  is identical to that of Levant \cite{Levant2001} for $1\leq r\leq 2$. The advantage of our controller in these cases is that the Lyapunov function provides an analytical method of tuning the controller parameters, whereas the tuning is empiric in Levant's case. Unfortunately this is not the case as soon as $r\geq 3$ and one needs the delicate analysis developed in  \cite{Levant2001}.

The amplitude of the discontinuous control given in Eq.~\eqref{toto0} is equal to $M = \left( l_r + \bar \varphi \right) / {\gamma_m}.$ We shall now see that this amplitude can be reduced to its minimum level $M_{min} = \bar \varphi / {\gamma_m}$
when the state $z$ tends to zero, by changing the degree of homogeneity. 

%{\color{red}
In the following theorem, both $\omega_\kappa(z)$ and $\bar \omega_\kappa(z)$, $\kappa\leq 0$ represent either $\omega_\kappa^{H}$ and $\omega_\kappa^{HM}$
and  the Lyapunov function $\bar V_{\kappa,r}$corresponds to $\omega_\kappa^{HM}$.
\begin{montheo}\label{prop2}
For $k \in (-1/r,0)$ and $A>0$ satisfying
%{\color{red}
\beqnum\label{Condition_Saturation}
	\mathop {\max }\limits_{\mathop { \bar V_{k,r}(z) \le A} } \left| \bar \omega_k(z) \right| \le l_r,

\eeqnum
we define the function
$
	U_{k,A}(z) \!\!:=\!\! \left\{ \begin{array}{ccc}
									\omega_{-1/r}(z) &\text{if}& \bar V_{k,r}(z) > A,\\
									\bar \omega_k(z) &\text{if}& \bar V_{k,r}(z) \le A.
					\end{array}  \right.\\
$
%}
Then the controller $u(z):=\left(U_{k,A}(z) + \bar \varphi sign(U_{k,A}(z))\right) / {\gamma_m}$ stabilizes System \eqref{u.l.s.} in finite time, and $u(z)$ is bounded with minimum amplitude of discontinuity $M_{min}$ at $z=0$.
\end{montheo}
\begin{proof}
%{\color{red}
Consider the following sets
$$
	\mathbf{S}_1= \{z \in \mathbb{R}^r: \left| \bar \omega_k(z) \right| \le l_r  \},\quad
	\mathbf{S}_2 = \{z \in \mathbb{R}^r:  \bar V_{k,r}(z) \le A  \}.
$$
According to Condition \eqref{Condition_Saturation}, we have $\mathbf{S}_2 \subset \mathbf{S}_1$.
As $\dot V_{-1/r,r}(z) < 0, \forall z \notin \mathbf{S}_2$, then every trajectory of System \eqref{u.l.s.} reaches $\mathbf{S}_2$ in finite-time.
Moreover, for $z\in \mathbf{S}_2$, $U_{k,A}(z)$ is equal to $\bar \omega_k(z)$, with $\left| \bar \omega_k(z) \right|\le l_r$. Therefore, as soon as a trajectory reaches $\mathbf{S}_2$, it will stay in it forever since $\dot {\bar V}_{k,r}(z)< 0, \forall z \notin \mathbf{S}_1,\forall z\ne 0$. One concludes that every trajectory of System \eqref{u.l.s.} converges to zero in finite-time and $ U_{k,A}(z) $ tends to zero as $z$ tends to zero. As a result, $ \forall z \in \mathbb{R}^r$, $\left |u(z) \right|  \le  M_{min} + {l_r} / {\gamma_m}.$ and 
$\lim_{\|z\| \to 0} \left |u(z) \right|  =  {\bar \varphi} / {\gamma_m} = M_{min}$.
\end{proof}
%%%%%%%%%%%%%%%%%%%%%%%%%%%%%%%%%%%%%%%%%%%%%%%%%%%%%%%%%%%%%%%%%%%%%%%%%%%%%%%%%%%%%%%%%%%%%%%%%%%%%
\subsection{Fixed-time Homogeneous controller}
In certain cases, it is required that the controller converges within a fixed interval of time, irrespectively of its initial condition. This can also be achieved by changing the homogeneity degree.\\
%{\color{red}
In the following theorem, $\bar \omega_\kappa(z)$, $\kappa\leq 0$ represents either $\omega_\kappa^{H}$ and $\omega_\kappa^{HM}$ and  $\bar V_{\kappa,r}$is  the corresponding Lyapunov function.
%
%
%$\bar \omega_\kappa(z)$ and $\omega_\kappa(z)$ represent two controllers, stabilizing \eqref{pure_integrators}, which may be identical or not. $\omega_\kappa(z)$ represents the result  $\omega_\kappa^{H}$ and $\omega_\kappa^{HN}$; and $\bar \omega_\kappa(z)$ represents the result $\omega_\kappa^{H}$, it's valid only for negative homogeneity degree . 
%We distinguish the parameters related to $\bar \omega_\kappa(z)$ by bar line ($\bar .$)
%}
\begin{montheo}
%{\color{red}
For $k_1 \in (0,+1/r)$, $k_2 \in (-1/r,0)$ and $B>0$, define

\beqnum\label{eq:a}
	E:=  \mathop {\min }\limits_{V_{k_2,r}(z) = B}  V_{k_1,r}(z) > 0,
	
\eeqnum
and the function
$
	U_{k,B}(z) = \left\{ \begin{array}{ccc}
									\omega_{k_1}^H(z) &\text{if}& V_{k_2,r}(z) > B,\\
									\bar\omega_{k_2}(z) &\text{if}& V_{k_2,r}(z) \le B.
					\end{array}  \right.\\
$
%}
Then the controller $u(z):=\left(U_{k,B}(z) + \bar \varphi sign(U_{k,B}(z)) \right) / {\gamma_m}$ stabilizes System \eqref{u.l.s.} in fixed time $T\leq T_u + T_f$ where the values of $T_u$ and $T_f$ are given by
%{\color{red}
$$
	T_u = {(2 + k_1)E^{\frac{k_1}{2+k_1 }}} / \left({ k_1 C }\right), \quad
	T_f = \left\{
						\begin{array}{lcl}
									{(2+k_2)B^{\frac{-k_2}{2+k_2}}} / \left({ -k_2 C }\right), \text{ if } \bar\omega_{k_2}(z) = \omega_\kappa^{H}(z).\\
									{(c+1)B^{\frac{-k_2}{c+1}}} / \left({ -k_2 C }\right), \text{ if } \bar\omega_{k_2}(z) = \omega_\kappa^{HM}(z).\\									
						\end{array}
				\right.
$$
%}
\end{montheo}

\begin{proof} The conclusion follows by integrating the differential equation $\dot V = -C V^\alpha$ on appropriate time intervals.
%{\color{red}
Consider first the following sets
$$	\mathbf{S}_1 = \{z \in \mathbb{R}^r: V_{k_1,r}( z) \le E  \},\quad
	\mathbf{S}_2 = \{z \in \mathbb{R}^r: \bar V_{k_2,r}( z) \le B  \}.
$$
According to Condition \eqref{eq:a}, we get that $\mathbf{S}_1 \subset \mathbf{S}_2$.
Clearly, $z$ will reach $\mathbf{S}_2$ in a fixed-time, bounded by a constant $T_u$, calculated as follows:
for $\alpha = 1 + \frac{k_1}{2+k_1}, \quad \int^{+\infty}_{E} {\frac{dV}{V^{\alpha}}} = -C \int^{T_u}_{0} {dt}$, then 	$T_u = {(2 + k_1)E^{\frac{k_1}{2+k_1 }}} / \left({ k_1 C }\right)$.
When $z$ reaches $\mathbf{S}_2$, i.e. $\bar V_{k_2,r} (z) = B$, $z$ will converge to zero in a finite-time bounded by $T_f$, which is calculated as follows:
for $\alpha = 1 + \frac{k_2}{2+k_2}, \quad \int^{0}_{B} {\frac{dV}{V^{\alpha}}} = -C \int^{T=T_u+T_f}_{T_u} {dt}$, then $T_f = {(2+k_2)B^{\frac{-k_2}{2+k_2}}} / \left({ -k_2C }\right)$. Finally, for $\alpha = 1 + \frac{k_2}{c+1}, \quad \int^{0}_{B} {\frac{dV}{V^{\alpha}}} = -C \int^{T=T_u+T_f}_{T_u} {dt}$, then $T_f = {(c+1)B^{\frac{-k_2}{c+1}}} / \left({ -k_2C }\right)$
%}
\end{proof}
%%%%%%%%%%%%%%%%%%%%%%%%%%%%%%%%%
\begin{rem}
The rate of convergence can be accelerated via time-rescaling (see Theorem 2 of Hong et al. \cite{Hong2005}). This is done by replacing the controller $\omega(z_1,z_2, \cdots, z_r)$ by $\bar \omega(z_1,z_2, \cdots, z_r) = \tau^r \omega(z_1,\frac{z_2}{\tau}, \cdots, \frac{z_r}{\tau^{r-1}}) $ where $\tau>1$, and taking  $u$ as $u = \frac{m}{\gamma_m}( \bar \omega + n \bar \varphi sign(\bar \omega))$. By taking $\bar t = \tau t$ and $\bar z_i = \tau^{1-i} z_i$, we obtain $\dot V(\bar z_1,\cdots,\bar z_r) \le -\tau C V(\bar z_1,\cdots,\bar z_r)$
and the settling time becomes $\bar T\leq (T_u + T_f)/\tau$.
%where $\bar T_u = {(2-k)E^{\frac{k}{2-k }}} / \left({ -k \tau C }\right), \bar T_f = {(2+k)B^{\frac{-k}{2+k}}} / \left({ -k \tau C }\right).$
%{\color{red}
%where,\\
%$
%	\bar T_u = {(2 + k_1)E^{\frac{k_1}{2+k_1 }}} / \left({ k_1 \tau C }\right), \ 
%	\bar T_f = \left\{
%						\begin{array}{lcl}
%									{(2+k_2)B^{\frac{-k_2}{2+k_2}}} / \left({ -k_2 \tau C }\right) &,& \text{if } \bar\omega_{k_2}(z) = \omega_\kappa^{H}(z).\\
%									{(c+1)B^{\frac{-k_2}{c+1}}} / \left({ -k_2 \tau C }\right) &,& \text{if } \bar\omega_{k_2}(z) = \omega_\kappa^{HM}(z).\\									
%						\end{array}
%				\right.
%$
%%}

\end{rem}

%%%%%%%%%%%%%%%%%%%%%%%%%%%%%%%%%%%%%%%%%%%%%%%%%%%%%%%%%%%%%%%%%%%%%
%%%%%%%%%%%%%%%%%%%%%%%%%%%%%%
\section{Simulation Results}
In this section, we illustrate the performance of our proposed controllers using the following perturbed triple integrator defined by:
%\beq
%	\left\{
$
		\begin{array}{ccc}
			\dot z_1 = z_2,\quad
			\dot z_2 = z_3,\quad
			\dot z_3 = \varphi + \gamma u,
		\end{array}
$
%	\right.
%\eeq
with $\varphi = \sin(t)$ and $\gamma = 3 + \cos(t)$. Then, we have
%\beq
$
	\gamma_m = 2,\quad \gamma_M = 4, \quad \bar \varphi = 1.
$\\
%\eeq
The parameters of the controller are chosen as follows:
%\beq
$
	l_1 = 1,\quad l_2 = 3,\quad l_3 = 10.\\
$
%\eeq
We start first by fixing the parameter $\kappa$ for different values $\{{1}/{4}, -{1}/{4}, -{1}/{3} \}$.\\
%, as follows
%\beq
%	\left\{
%		\begin{array}{ccccc}
%			k &=& \kappa_u &=& \frac{1}{3},\\
%			%k &=& \kappa_a &=& 0,\\
%			k &=& \kappa_{f1} &=& -\frac{1}{8},\\
%			k &=& \kappa_{f2} &=& -\frac{1}{3}.
%		\end{array}
%	\right.
%\eeq
%Then we show the performance of the bounded controller with minimum discontinuity by switching $k$ from $\kappa_{f2}$ to $\kappa_{f1}$. Then we show the uniform controller with bounded time convergence time independent on initial condition.\\
For $\kappa>0$, Figure \ref{fig_u} shows a fast convergence of the states to a neighborhood of zero by an unbounded controller, otherwise the convergence to zero is asymptotic. For $-1/3<\kappa<0$, the convergence of the states to zero in finite-time is obtained by an unbounded controller with a minimum amplitude of the discontinuous control at $z=0$, as shown in Figure \ref{fig_f1}. The finite-time convergence of the states is also shown in Figure \ref{fig_f2} for $\kappa=-1/3$, using a bounded controller with a large discontinuous control at $z=0$.\\
The performance of a bounded controller which ensures a minimum discontinuous control amplitude at zero is shown in Figure \ref{fig_b} by switching $\kappa$ in neighborhood of zero, from $-1/3$ to $-1/4$.\\
The performance of a fixed-time controller is shown in Figure \ref{fig_x}. Figure \ref{time_x} shows that the convergence time will not exceed $8.5$ sec for any initial condition. Fixed-time stability is assumed to be established by the time after which, $|z_1|,\ |z_2|, \ |z_3|$ are less than $ 1 \times 10^{-4}$.

%----------------------------------------------------------
\begin{figure}[htbp]
\centering
\subfigure[control law $u$ versus time ($s$).]{	
    \includegraphics[width= 7.5 cm, height = 2.9 cm]{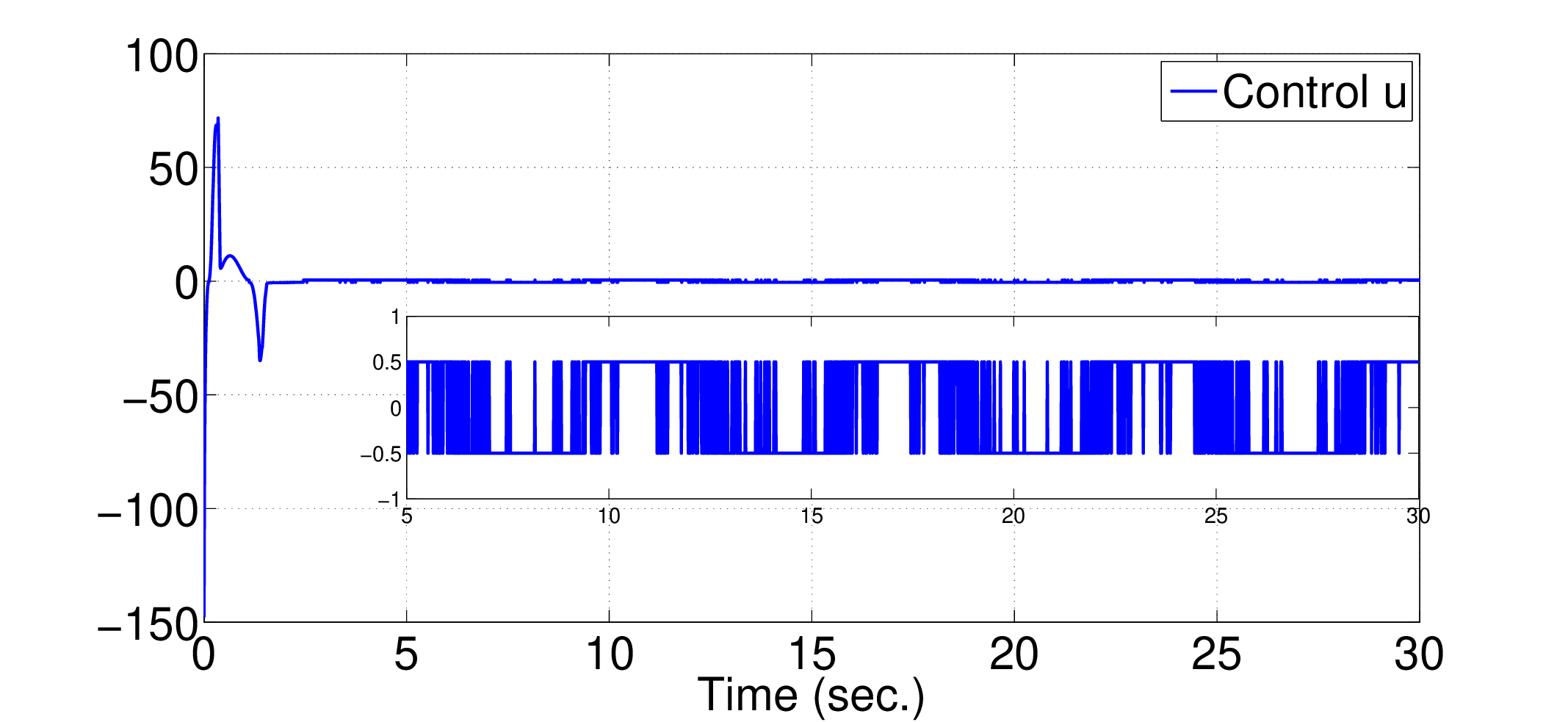}
    \label{u_u}
}%\hspace{-0.5cm}
\subfigure[$z_1$ and $z_2$ versus time ($s$).]{	
    \includegraphics[width= 7.5 cm, height = 2.9 cm]{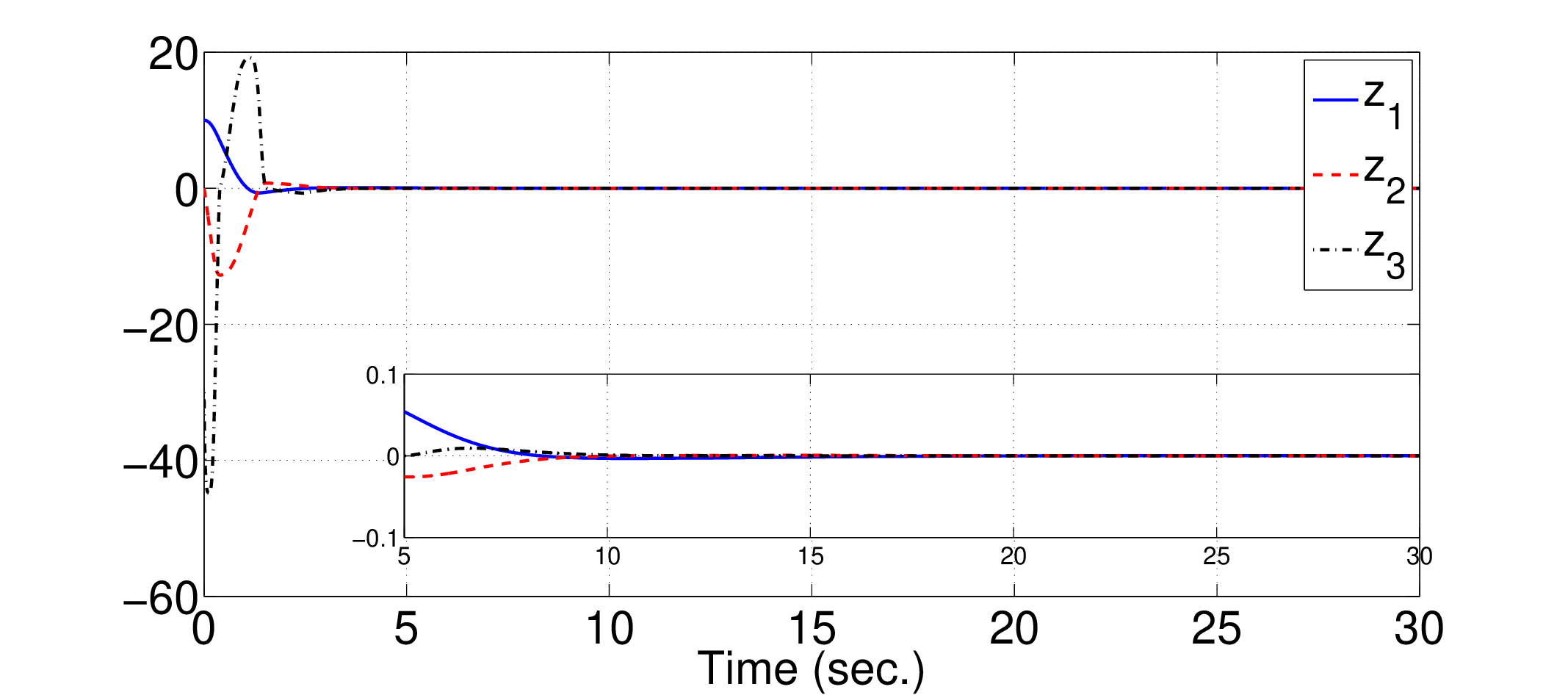}
    \label{states_u}
}
\caption{test for $\kappa>0$}
\label{fig_u}
\centering%\hspace{-0.5cm}
\subfigure[control law $u$ versus time ($s$).]{
    \includegraphics[width= 7.5 cm, height = 2.9 cm]{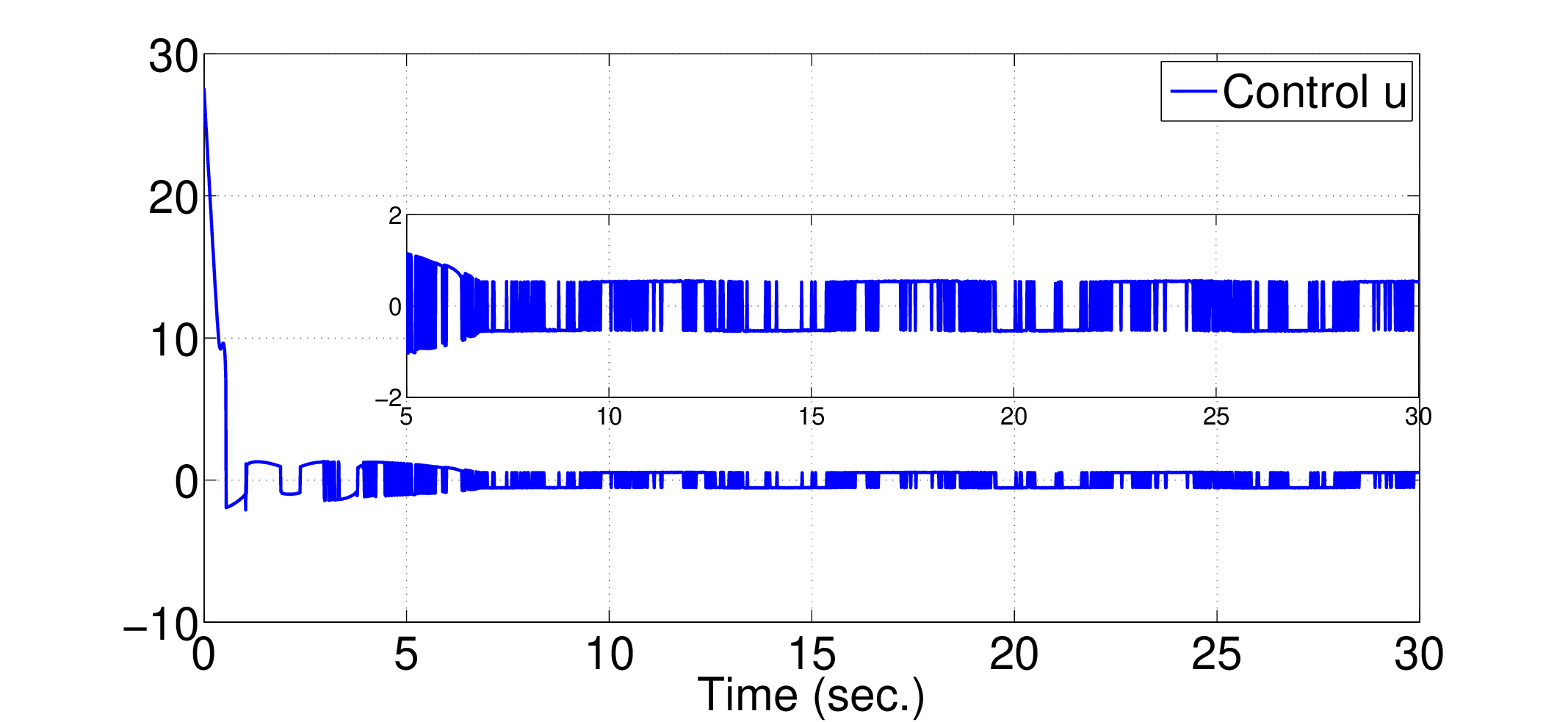}
    \label{u_f1}
}%\hspace{-0.5cm}
\subfigure[$z_1$ and $z_2$ versus time ($s$).]{
    \includegraphics[width= 7.5 cm, height = 2.9 cm]{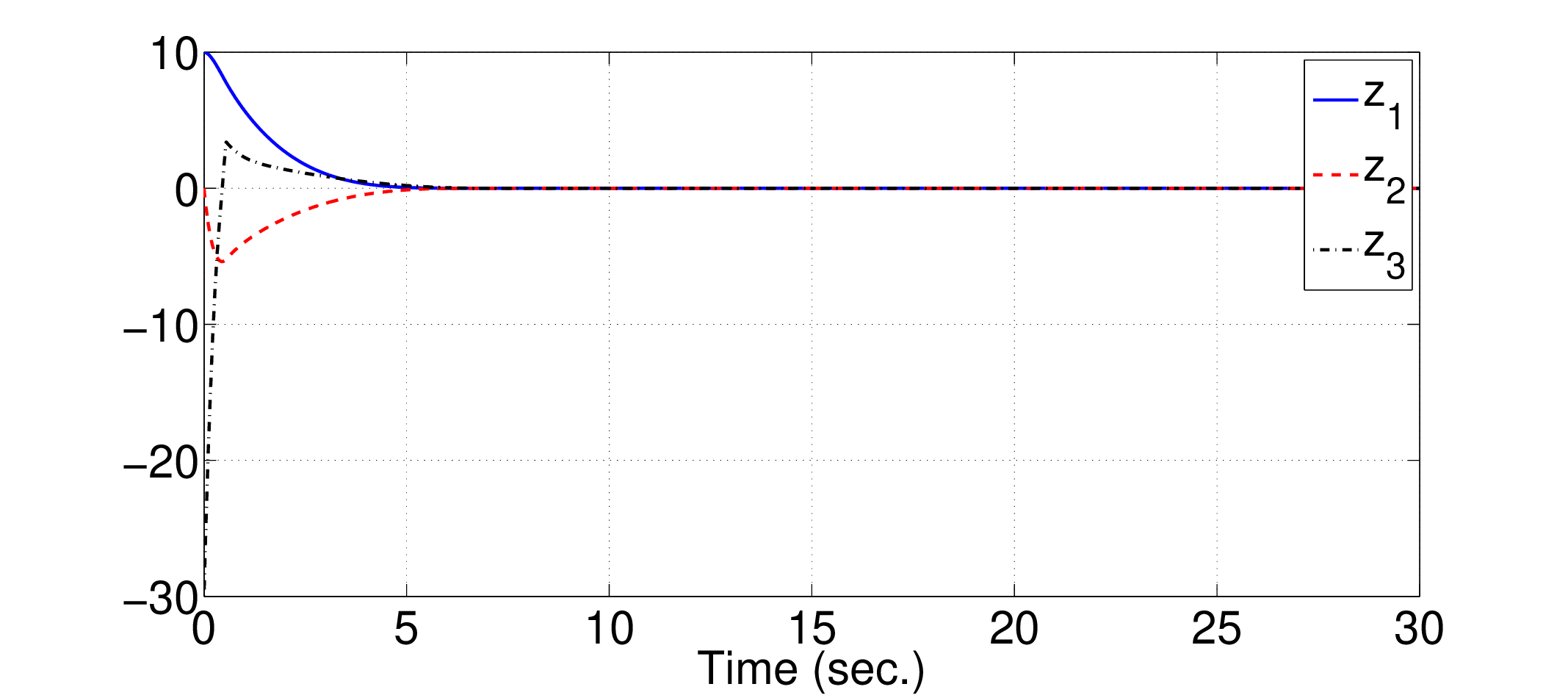}
    \label{states_f1}
}
\caption{test for $-1/r<\kappa<0$}
\label{fig_f1}

\centering%\hspace{-0.5cm}
\subfigure[control law $u$ versus time ($s$).]{
    \includegraphics[width= 7.5 cm, height = 2.9 cm]{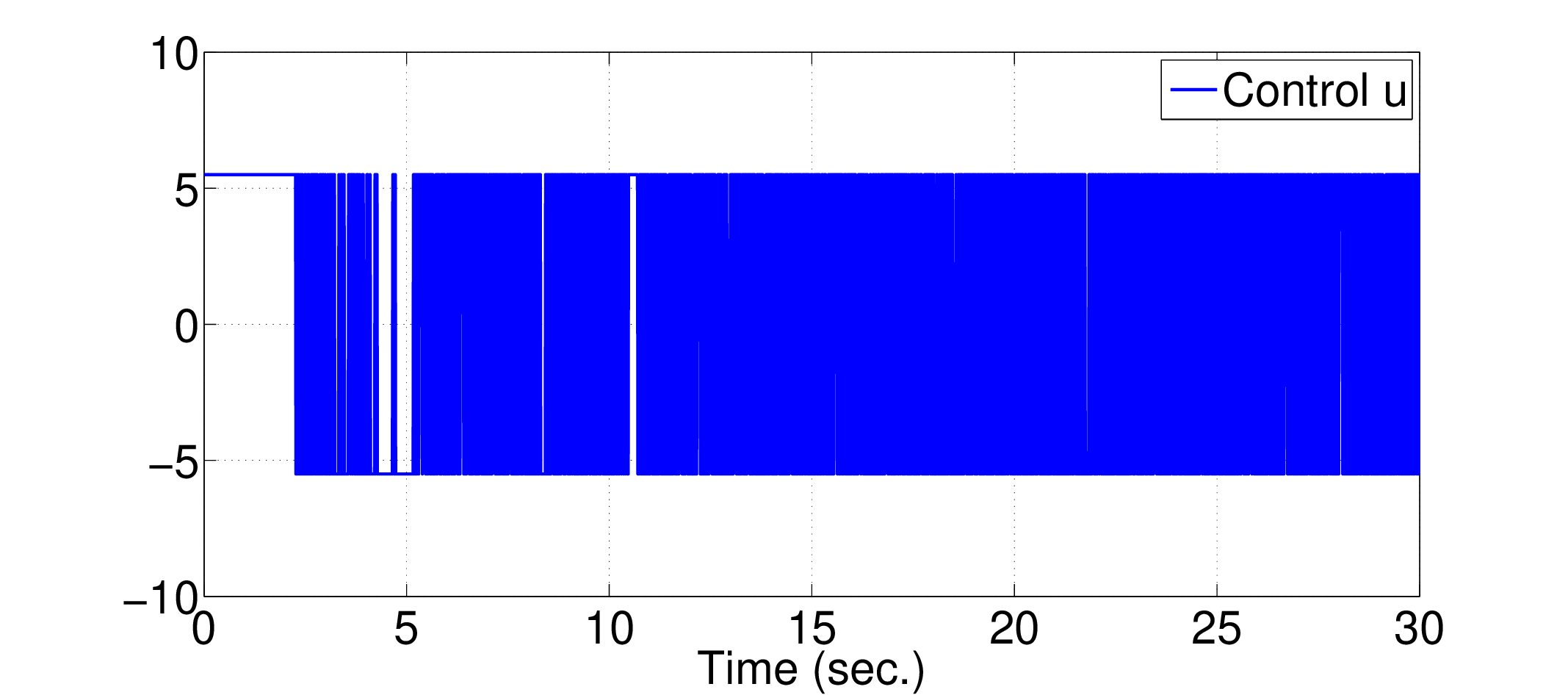}
    \label{u_f2}
}%\hspace{-0.5cm}
\subfigure[$z_1$ and $z_2$ versus time ($s$).]{
    \includegraphics[width= 7.5 cm, height = 2.9 cm]{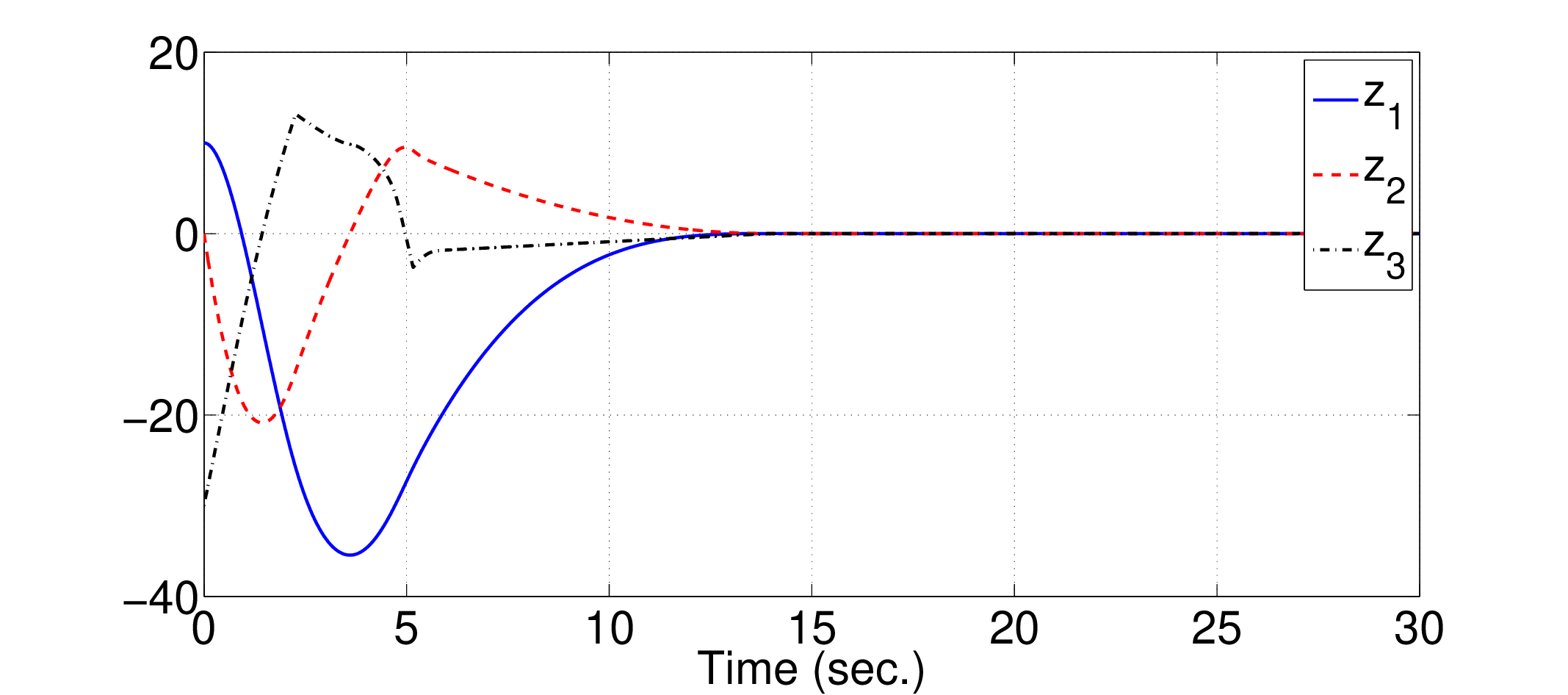}
    \label{states_f2}
}
\caption{test for $\kappa = -1/r$ (case equivalent of \cite{Levant2001})}
\label{fig_f2}
%\end{figure}
%
%%%%%%%%%%%%%%%%%%%%%%%%%%%%%%%%%%%%%%%%%%%%%%%%%%%%%%%%%%%%%%%
%
%\begin{figure}[htbp]
\centering%\hspace{-0.5cm}
\subfigure[control law $u$ versus time ($s$).]{
    \includegraphics[width= 7.5 cm, height = 2.9 cm]{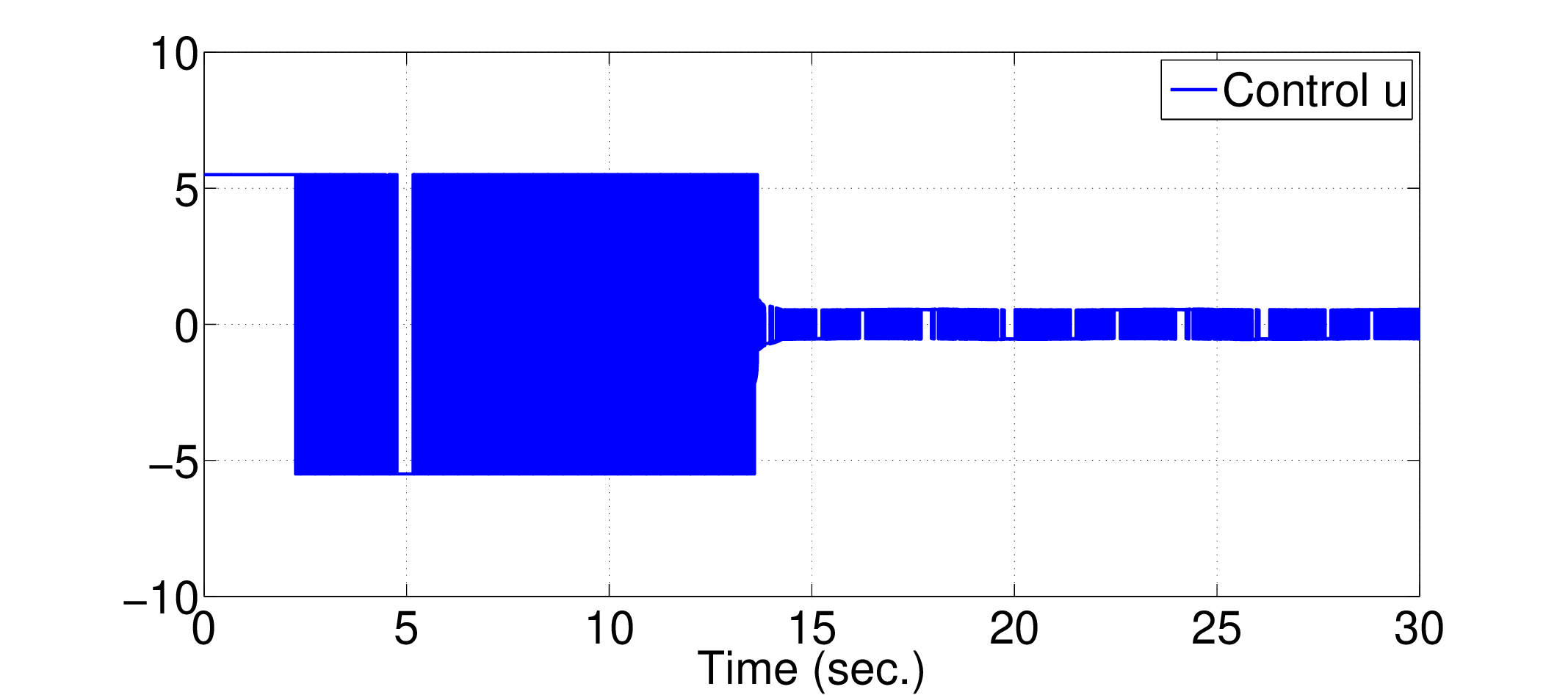}
    \label{u_b}
}%\hspace{-0.5cm}
\subfigure[$z_1$ and $z_2$ versus time ($s$).]{
    \includegraphics[width= 7.5 cm, height = 2.9 cm]{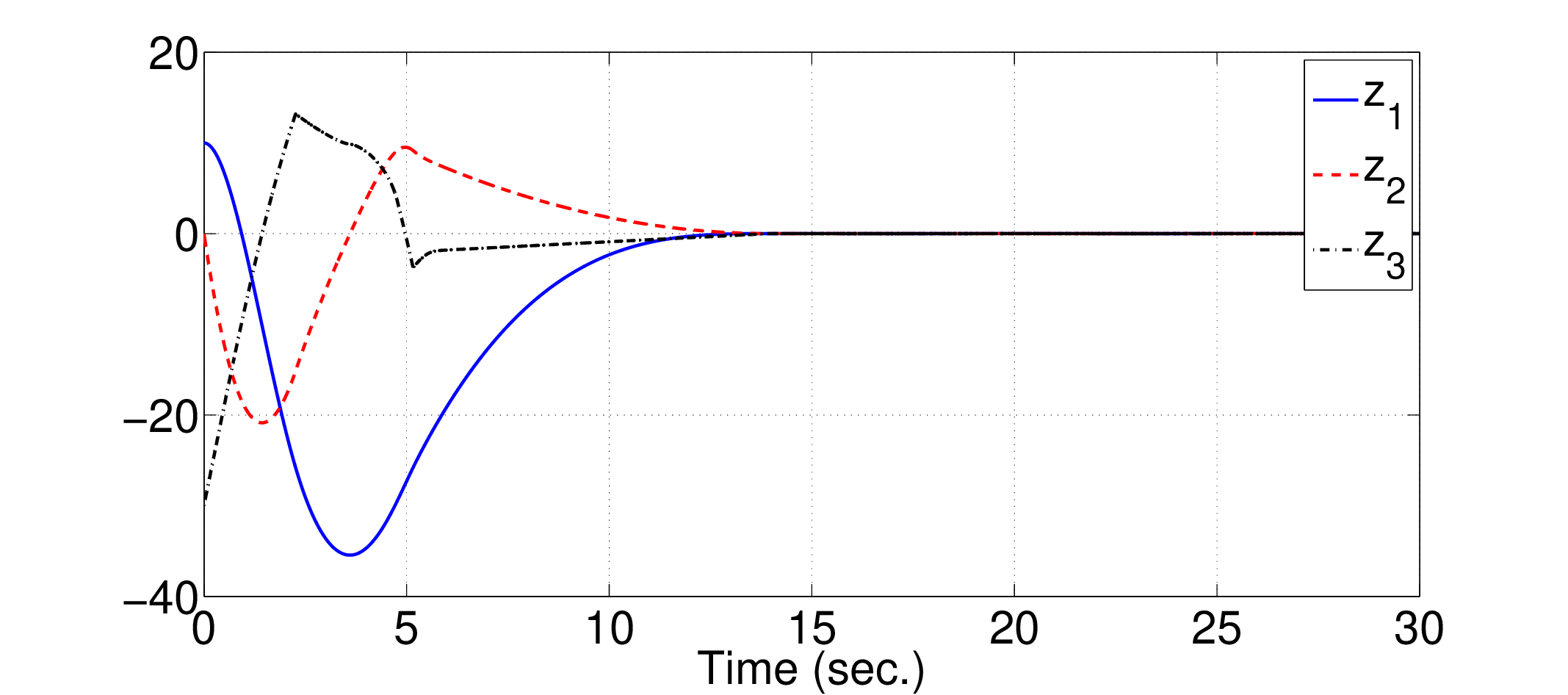}
    \label{states_b}
}
\caption{test for $\kappa$ switching from $ -1/r$ to $ k \in (-1/r,0)$}
\label{fig_b}
%\end{figure}
%\vspace{-0.5cm}
%\begin{figure}[htbp]
\centering%\hspace{-0.5cm}
\subfigure[control law $u$ versus time ($s$).]{
    \includegraphics[width= 7.5 cm, height = 2.9 cm]{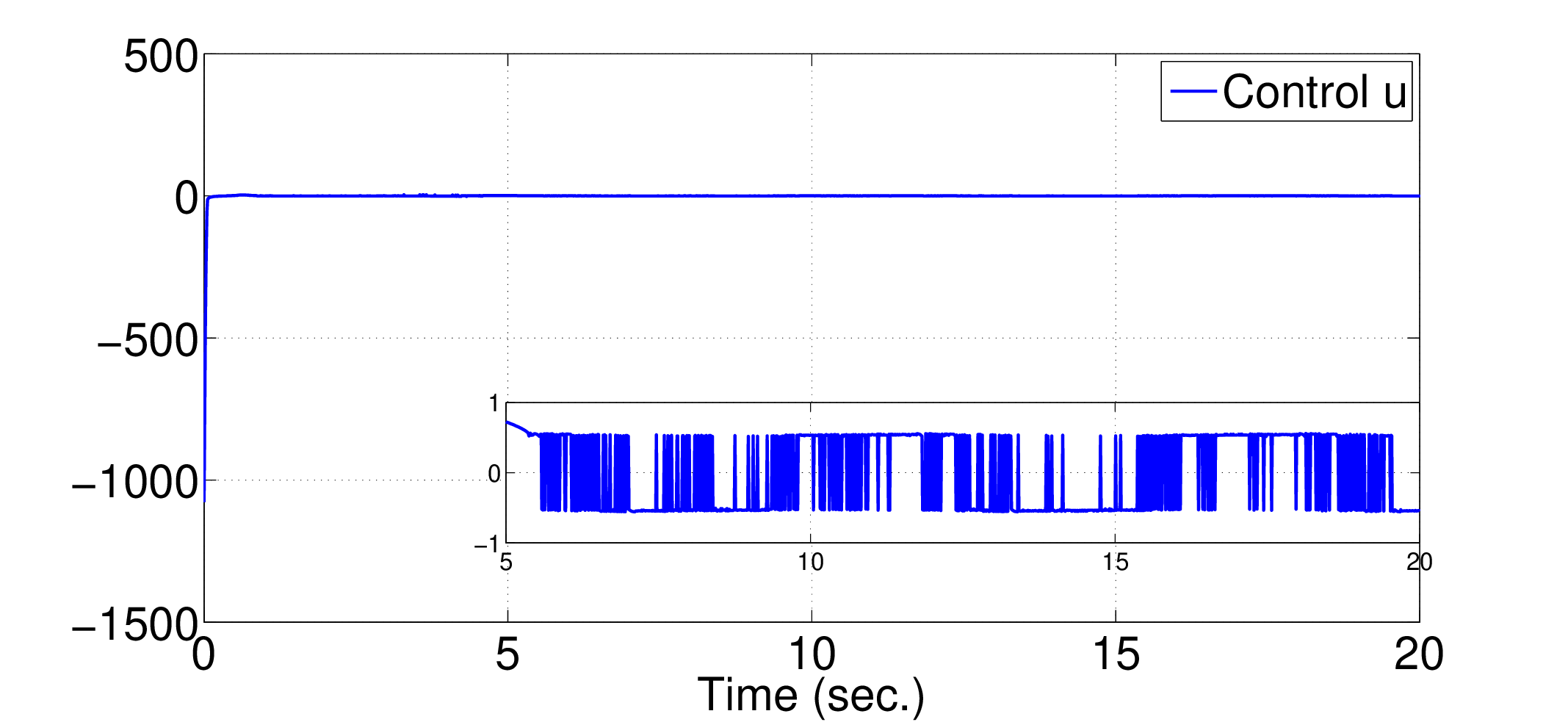}
    \label{u_x}
}%\hspace{-0.5cm}
\subfigure[$z_1$ and $z_2$ versus time ($s$).]{
    \includegraphics[width= 7.5 cm, height = 2.9 cm]{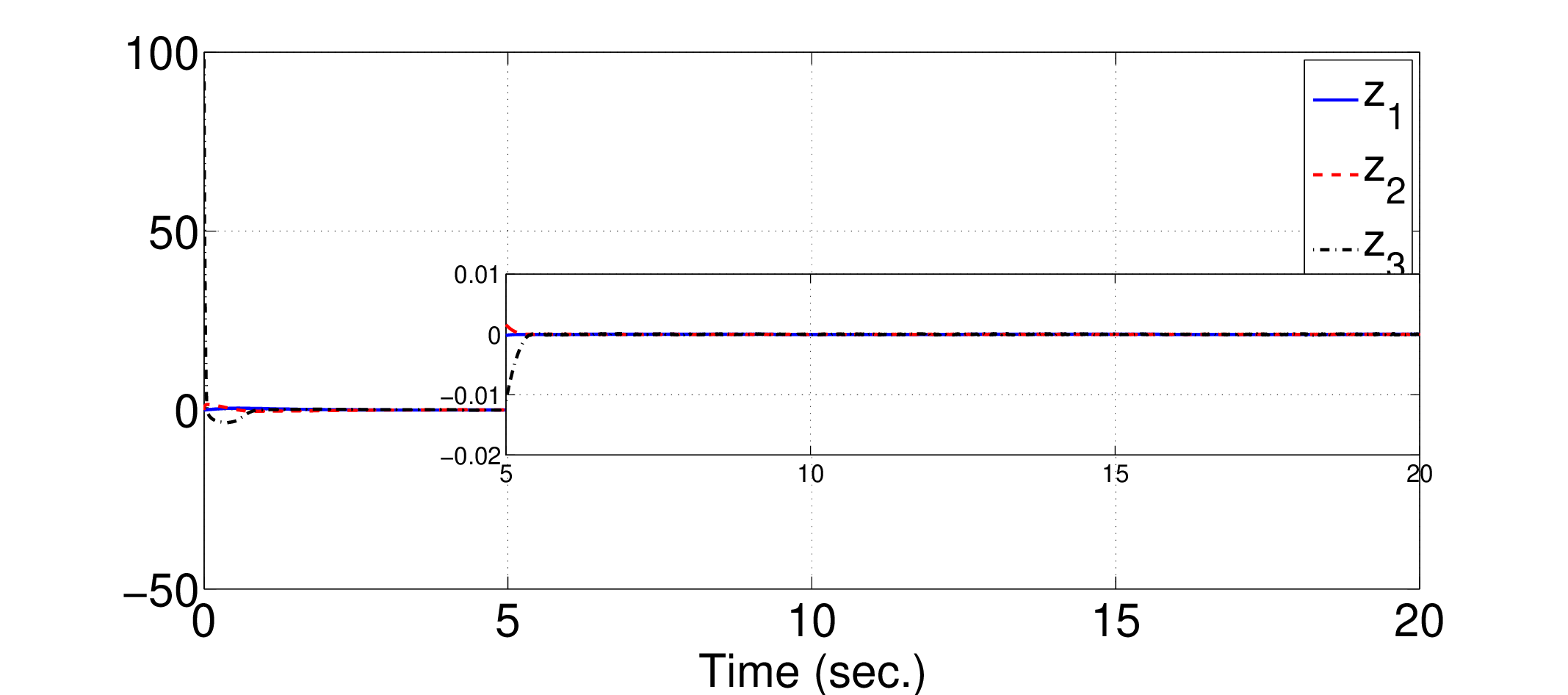}
    \label{states_x}
}
\caption{test for $\kappa$ switching from $-k$ to $k$, $k \in (-1/r,0)$ }
\label{fig_x}

\end{figure}

%-----------------------------------------
\begin{figure}[htbp]
\begin{center}
\includegraphics[width=10cm]{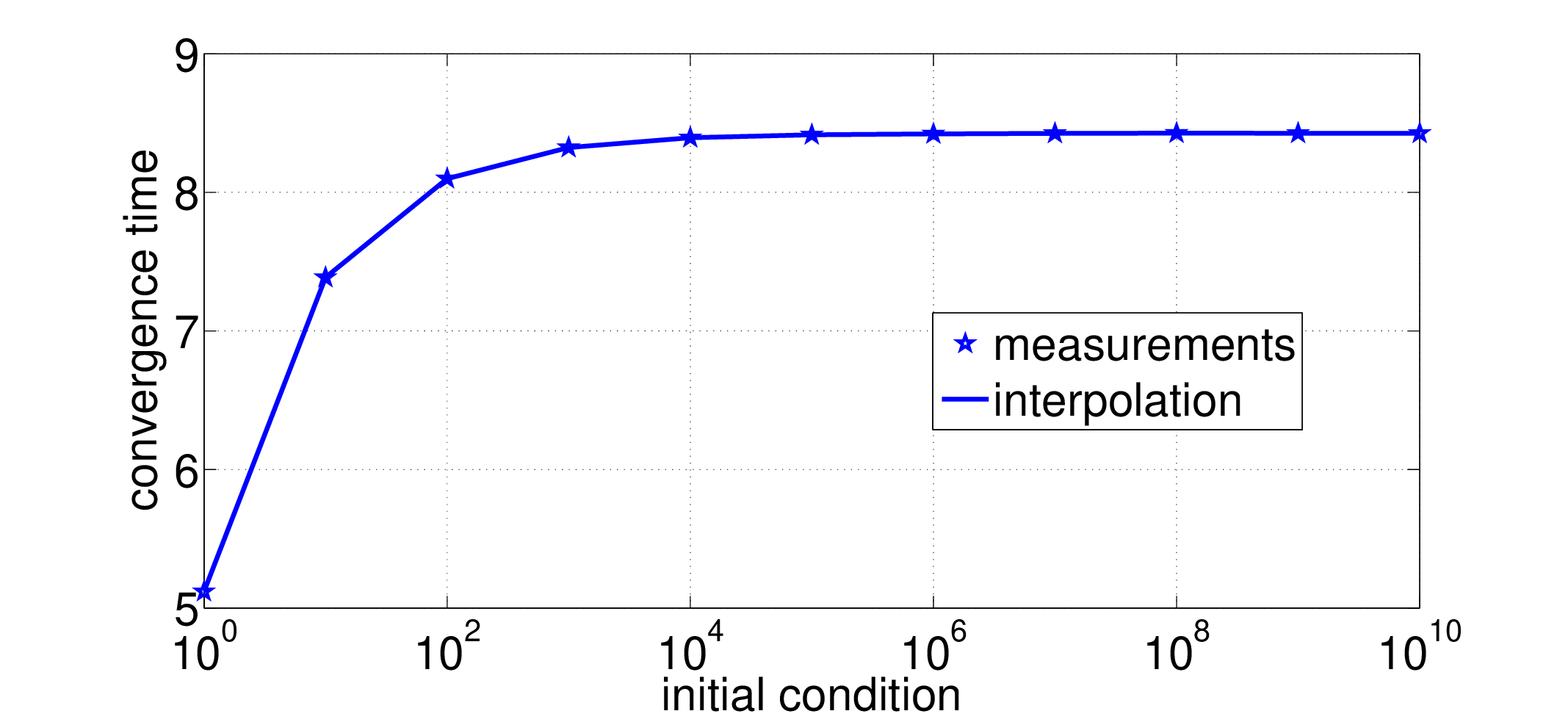}
\caption{Convergence time versus initial condition.}
\label{time_x}
\end{center}
\end{figure}
%-----------------------------------------
%
%\vspace{-1cm}
\section{Conclusions}
In this paper, we presented a Lyapunov-based method for designing finite-time convergent controllers for stabilization of perturbed integrator chains of arbitrary order. This method consists in appropriate modifications of  homogeneous controller stabilizing pure integrator chains. It was also shown that the properties of minimum discontinuity amplitude of the controller and fixed-time convergence can be obtained by changing the homogeneity degree of the controller. %\section{graph}
%-------------------------------------------------------------------
%\pagebreak
\bibliographystyle{IEEEtran}
\bibliography{Ref}

\begin{thebibliography}{10}
\providecommand{\url}[1]{#1}
\csname url@samestyle\endcsname
\providecommand{\newblock}{\relax}
\providecommand{\bibinfo}[2]{#2}
\providecommand{\BIBentrySTDinterwordspacing}{\spaceskip=0pt\relax}
\providecommand{\BIBentryALTinterwordstretchfactor}{4}
\providecommand{\BIBentryALTinterwordspacing}{\spaceskip=\fontdimen2\font plus
\BIBentryALTinterwordstretchfactor\fontdimen3\font minus
  \fontdimen4\font\relax}
\providecommand{\BIBforeignlanguage}[2]{{%
\expandafter\ifx\csname l@#1\endcsname\relax
\typeout{** WARNING: IEEEtran.bst: No hyphenation pattern has been}%
\typeout{** loaded for the language `#1'. Using the pattern for}%
\typeout{** the default language instead.}%
\else
\language=\csname l@#1\endcsname
\fi
#2}}
\providecommand{\BIBdecl}{\relax}
\BIBdecl

\bibitem{Baht1998}
S.~Bhat and D.~Bernstein, ``Continuous finite-time stabilization of the
  translational and rotational double integrators,'' \emph{Automatic Control,
  IEEE Transactions on}, vol.~43, no.~5, pp. 678--682, 1998.

\bibitem{Hong2002_Robot}
Y.~Hong, Y.~Xu, and J.~Huang, ``Finite-time control for robot manipulators,''
  \emph{Systems and Control Letters}, vol.~46, no.~4, pp. 243 -- 253, 2002.

\bibitem{Orlov2009}
Y.~Orlov, \emph{Discontinuous Systems - Lyapunov Analysis and Robust Synthesis
  Under Uncertainty Conditions}.\hskip 1em plus 0.5em minus 0.4em\relax London,
  U.K.: Springer-Verlag, 2009.

\bibitem{Emelyanove}
S.~Emel'yanov, S.~Korovin, and A.~Levant, ``High-order sliding modes in control
  systems,'' \emph{Computational Mathematics and Modeling}, vol.~7, no.~3, pp.
  294--318, 1996.

\bibitem{Dinuzzo09}
F.~Dinuzzo and A.~Fererra, ``{Higher Order Sliding Mode Controllers with
  Optimal Reaching},'' \emph{IEEE Transactions on Automatic Control}, vol.~54,
  no.~9, pp. 2126--2136, 2009.

\bibitem{Bhat_Bernstein_1}
S.~Bhat and D.~Bernstein, ``Finite-time stability of homogeneous systems,'' in
  \emph{American Control Conference, 1997. Proceedings of the 1997}, vol.~4,
  pp. 2513--2514.

\bibitem{Praly_1997}
L.~Praly, ``Generalized weighted homogeneity and state dependent time scale for
  linear controllable systems,'' in \emph{Decision and Control, 1997.,
  Proceedings of the 36th IEEE Conference on}, vol.~5, pp. 4342--4347.

\bibitem{Bernstein2005}
S.~Bhat and D.~Bernstein, ``Geometric homogeneity with applications to
  finite-time stability,'' \emph{Math. Control Signals Systems}, vol.~17, pp.
  101 -- 127, 2005.

\bibitem{Hong}
Y.~Hong, ``Finite-time stabilization and stabilizability of a class of
  controllable systems,'' \emph{Systems and Control Letters}, vol.~46, no.~4,
  pp. 231--236, 2002.

\bibitem{Qian_1}
J.~Li and C.~Qian, ``Global finite-time stabilization of a class of uncertain
  nonlinear systems using output feedback,'' in \emph{Decision and Control,
  2005 and 2005 European Control Conference. CDC-ECC '05. 44th IEEE Conference
  on}, pp. 2652--2657.

\bibitem{Qian_2}
C.~Qian, ``A homogeneous domination approach for global output feedback
  stabilization of a class of nonlinear systems,'' in \emph{American Control
  Conference, 2005. Proceedings of the 2005}, 2005, pp. 4708--4715.

\bibitem{Levant2003}
A.~Levant, ``Higher-order sliding modes, differentiation and output-feedback
  control,'' \emph{International Journal of Control}, vol.~76, no. 9/10, pp.
  924 -- 941, 2003.

\bibitem{Levant2005}
------, ``Homogeneity approach to high-order sliding mode design,''
  \emph{Automatica}, vol.~41, no.~5, pp. 823 -- 830, 2005.

\bibitem{Levant2001}
------, ``{Universal Single-Input–Single-Output (SISO) Sliding-Mode Controllers
  With Finite-Time Convergence},'' \emph{IEEE Transactions on Automatic
  Control}, vol.~46, no.~9, pp. 1447 -- 1451, 2001.

\bibitem{Defoort2009}
M.~Defoort, T.~Floquet, A.~Kokosy, and W.~Perruquetti, ``A novel higher order
  sliding mode control scheme,'' \emph{Systems and Control Letters}, vol.~58,
  no.~2, pp. 102 -- 108, 2009.

\bibitem{Kryachkov}
M.~Kryachkov, A.~Polyakov, and V.~Strygin, ``Finite-time stabilization of an
  integrator chain using only signs of the state variables,'' in \emph{Variable
  Structure Systems (VSS), 2010 11th International Workshop on}, pp. 510--515.

\bibitem{Harmouche_CDC12}
M.~Harmouche, S.~Laghrouche, and Y.~Chitour, ``Robust and adaptive higher order
  sliding mode controllers,'' in \emph{Decision and Control (CDC), 2012 IEEE
  51st Annual Conference on}, pp. 6436--6441.

\bibitem{Andrieu2008}
V.~Andrieu, L.~Praly, and A.~Astolfi, ``Homogeneous approximation, recursive
  observer and output feedback,'' \emph{SIAM Journal of Control and
  Optimization}, vol.~47, no.~4, pp. 1814--1850, 2008.

\bibitem{Zavala_2011}
E.~Cruz-Zavala, J.~Moreno, and L.~Fridman, ``Second-order uniform exact sliding
  mode control with uniform sliding surface,'' in \emph{Decision and Control
  and European Control Conference (CDC-ECC), 2011 50th IEEE Conference on}, pp.
  4616--4621.

\bibitem{Zavala_2012}
------, ``Uniform sliding mode controllers and uniform sliding surfaces,''
  \emph{IMA Journal of Mathematical Control and Information}, vol.~39, no.~4,
  pp. 491 -- 505, 2012.

\bibitem{Polyakov}
A.~Polyakov, ``Nonlinear feedback design for fixed-time stabilization of linear
  control systems,'' \emph{Automatic Control, IEEE Transactions on}, vol.~57,
  no.~8, pp. 2106 --2110, 2012.

\bibitem{Isidori}
A.~Isidori, \emph{Nonlinear control systems: An introduction (3rd ed.)}.\hskip
  1em plus 0.5em minus 0.4em\relax Berlin: Springer, 1995.

\bibitem{Filippov}
A.~Filippov, \emph{Differential Equations with Discontinuous Right-Hand
  Side}.\hskip 1em plus 0.5em minus 0.4em\relax Dordrecht, The Netherlands:
  Kluwer, 1988.

\bibitem{Bernstein2000}
S.~Bhat and D.~Bernstein, ``Finite-time stability of continuous autonomous
  systems,'' \emph{SIAM Journal of Control and Optimization}, vol.~38, no.~3,
  pp. 751 -- 766, 2000.

\bibitem{Zhang2013}
J.~Zhang, Z.~Han, and J.~Huang, ``Homogeneous feedback design of differential
  inclusions based on control lyapunov functions,'' \emph{Communications in
  Nonlinear Science and Numerical Simulation}, vol.~18, no.~10, pp. 2790 --
  2800, 2013.

\bibitem{Rosier1992}
L.~Rosier, ``Homogeneous lyapunov function for homogeneous continuous vector
  field,'' \emph{Systems and Control Letters}, vol.~19, no.~6, pp. 467 -- 473,
  1992.

\bibitem{Hong2005}
Y.~Hong, J.~Wang, and Z.~Xi, ``Stabilization of uncertain chained form systems
  within finite settling time,'' \emph{IEEE Transactions on Automatic Control},
  vol.~50, no.~9, pp. 1379--1384, 2005.

\end{thebibliography}
\end{document}